\documentclass[11pt,reqno]{amsart}
\pdfoutput=1
\usepackage{geometry}   
\usepackage{graphicx}
\usepackage{amssymb}
\usepackage{epstopdf} 
\usepackage{hyperref}
\usepackage{graphicx, amsmath, amsthm, latexsym, amssymb, amsfonts, epsfig, url, paralist}
\usepackage{xcolor}
\usepackage{tikz}
\usetikzlibrary{arrows}

\newtheorem{theorem}{Theorem}[section]
\newtheorem{lemma}[theorem]{Lemma}
\newtheorem{proposition}[theorem]{Proposition}
\newtheorem{corollary}[theorem]{Corollary}

\theoremstyle{definition}
\newtheorem{definition}[theorem]{Definition}
\newtheorem{remark}[theorem]{Remark}
\newtheorem{question}[theorem]{Question}

\newcommand\R{\mathbb R}
\newcommand\GKZ{{\rm GKZ}_n}
\newcommand\Post{{\rm Post}_n}
\newcommand\RSS{{\rm RSS}_n}
\newcommand\CFZ{{\rm CFZ}_n}
\newcommand\AssI{{\rm Ass}_n^{\rm I}}
\newcommand\AssII{{\rm Ass}_n^{\rm II}}
\newcommand{\conv}{\operatorname{conv}}

\title[Many non-equivalent realizations of the associahedron]{Many non-equivalent realizations\\of the associahedron}

\author{Cesar Ceballos}
\address[Cesar Ceballos]{Department of Mathematics and Statistics,
York University,
4700 Keele St, Toronto, ON M3J 1P3, Canada.}  \email{ceballos@mathstat.yorku.ca}

\author{Francisco Santos}
\address[Francisco Santos]{Facultad de Ciencias, Universidad de Cantabria, 
Av. de los Castros s/n, E-39005 Santander, Spain.}  
\email{francisco.santos@unican.es}

\author{G\"unter M. Ziegler}
\address[G\"unter M. Ziegler]{Inst.\ Mathematics, FU Berlin, Arnimallee 2, 14195 Berlin, Germany.} \email{ziegler@math.fu-berlin.de}
  
\thanks{The first author was supported by DFG via the Research Training Group ``Methods for Discrete Structures'' and the Berlin Mathematical School;
the second author was partially supported partially supported by the Spanish Ministry of Science under grants MTM2008-04699-C03-02, MTM2011-22792 and CSD2006-00032 (i-MATH) and 
by MICINN-ESF EUROCORES programme EuroGIGA -- ComPoSe -- IP04 (Project EUI-EURC-2011-4306);
the third author was partially supported by DFG and by ERC Advanced Grant ``SDModels'' (agreement no.~247029).
We are grateful to Anton Dochterman, Vincent Pilaud, and in particular Carsten Lange for 
helpful discussions and comments. 
We also very much appreciate extensive critical comments by two referees: The revision based on their recommendations has substantially improved this paper.}

\begin{document}

\begin{abstract}
Hohlweg and Lange (2007) and Santos (2004, unpublished) have found two different ways of constructing exponential families of realizations of the $n$-dimensional associahedron with normal vectors in $\{0,\pm 1\}^n$, generalizing the constructions of Loday (2004) and Chapoton--Fomin--Zelevinsky (2002). We classify the associahedra obtained by these constructions modulo linear equivalence of their normal fans and show, in particular, that the only realization that can be obtained with both methods is the Chapoton--Fomin--Zelevinsky (2002) associahedron.

For the Hohlweg--Lange associahedra our classification is a priori coarser than the classification up to isometry of normal fans, by Bergeron--Hohlweg--Lange--Thomas (2009). However, both yield the same classes. As a consequence, we get that two Hohlweg--Lange associahedra have linearly equivalent normal fans if and only if they are isometric. 

The Santos construction, which produces an even larger family of associahedra, appears here in print for the first time. 
Apart of describing it in detail we relate it with the $c$-cluster complexes and the denominator fans in cluster algebras of type $A$.

A third classical construction of the associahedron, as the secondary polytope of a convex $n$-gon (Gelfand--Kapranov--Zelevinsky, 1990), is shown to never produce a normal fan linearly equivalent to any of the other two constructions.
\end{abstract}

\maketitle\enlargethispage{3mm}\vspace{-9mm}

\tableofcontents

\section{Introduction}
\label{sec:intro}

The $n$-dimensional associahedron is a simple polytope with $C_{n+1}$ (the Catalan number)
 vertices, corresponding to the triangulations of a convex $(n+3)$-gon, and $n(n+3)/2$ facets,
 in bijection with the diagonals of the $(n+3)$-gon. 
It appears in Dov Tamari's unpublished 1951 thesis \cite{Tamari1951}, and was described as a combinatorial object
and realized as a cellular ball by Jim Stasheff in 1963 in his work on the associativity of $H$-spaces \cite{St63}. 
A realization as a polytope by John Milnor from the 1960s is lost; Huguet \& Tamari claimed in 1978
that the associahedron can be realized as a convex polytope \cite{HuguetTamari1978}.
The first such construction, 
via an explicit inequality system, 
was provided in a manuscript by Mark Haiman from 1984 that remained unpublished, but is available as \cite{Ha84}.
The first constructions in print are
due to Carl Lee, from 1989 \cite{Lee89}.
\begin{figure}[ht]
	\includegraphics[width=0.30\textwidth]
	{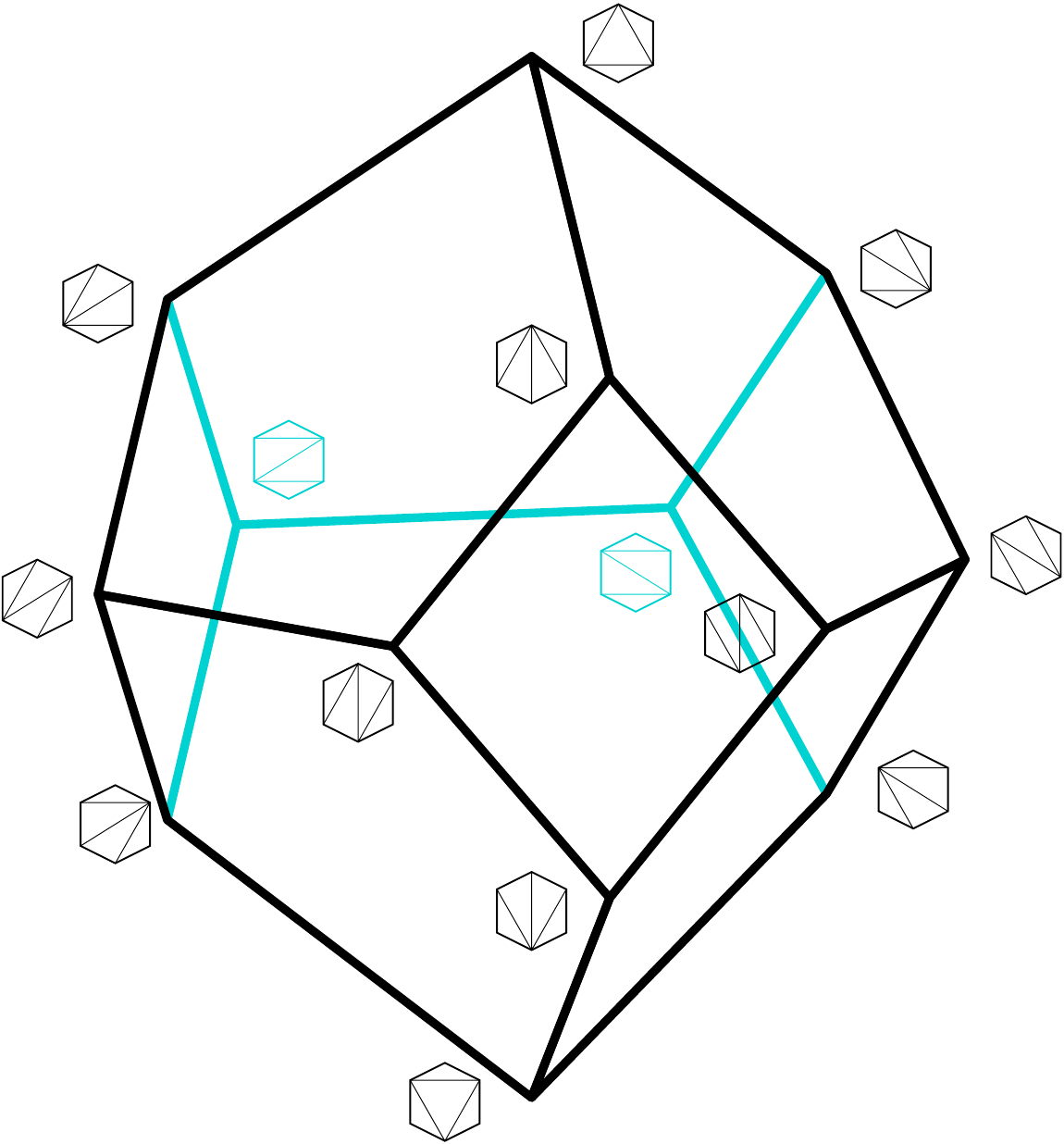}
	\caption{An associahedron,  as the secondary polytope of a regular hexagon.}
	\label{fig:hex-secondary}
\end{figure}

\noindent
Subsequently three systematic approaches were developed that produce realizations
of the associahedra in more general frameworks:
\begin{compactitem}[$\circ$]
	\item 
the associahedron as the \emph{secondary polytope} of a convex polygon, due to Gelfand, Kapranov and Zelevinsky \cite{GZK90, GZK91} (see also \cite[Chap.~7]{GKZ94}), depicted in Figure~\ref{fig:hex-secondary}.
	\item
the associahedron as one of the \emph{generalized permutahedra} introduced by Postnikov in \cite{Po05}. 
The history of this construction begins with Shnider and Sternberg \cite{ShniderSternberg93} 
(compare Stasheff and Shnider \cite[Appendix~B]{St97}), who show that the associahedron can be obtained by cutting certain faces in a simplex (this is  polar to the construction by Carl Lee in~\cite[Sec. 3]{Lee89}, which produces the normal fan of the associahedron as a stellar subdivision of the central fan of the simplex). Loday~\cite{Lo04}, shows how to obtain explicit and nice vertex coordinates for this associahedron using combinatorics of binary trees. Postnikov~\cite{Po05} then puts Loday's construction in context, regarding this associahedron as a special case of a \emph{generalized permutahedron}; 
a polytope lying in (the closure of) the deformation cone of the standard permutahedron.
Rote, Santos and Streinu~\cite{RoSaSt03} and, more recently, Buchstaber~\cite{Bu08} found constructions of essentially the same asociahedron but described quite differently. Following \cite{HoLa07,Po05} we reference this associahedron as the ``Loday realization''.
	\item
the associahedron as the polar of the \emph{cluster complex} of type $A_n$, conjectured by Fomin and Zelevinsky~\cite{FZ03} and constructed by Chapoton, Fomin and Zelevinsky \cite{CFZ02}.
\end{compactitem}

\noindent
We review these three constructions in Section~\ref{review}, after some preliminaries in Section~\ref{sec:preliminaries}. The last two of them have the following properties in common: 
\begin{compactenum}
\item They both have exactly $n$ pairs of parallel facets.
\item In the basis given by the normals to those $n$ pairs, all facet normals have coordinates in $\{0,\pm 1\}$.
\end{compactenum}

This was generalized by Hohlweg and Lange~\cite{HoLa07} and by Santos~\cite{Sa04}, who showed that the Loday and Chapoton--Fomin--Zelevinsky constructions are particular cases of \emph{exponentially many different} constructions of the associahedron, all of them with properties (1) and (2). 
That is, all these associahedra are (normally isomorphic to) polytopes obtained from the regular $n$-cube by cutting certain $\binom{n}{2}$ faces, as seen in Figure~\ref{fig:ManyAssociahedra}. 
Note, however, that the last example of Figure~\ref{fig:ManyAssociahedra} cannot be obtained by cutting faces one after the other; that is to say, its normal fan is not a stellar subdivision of the normal fan of the cube. 

\begin{figure}[ht]
	\centering
	\includegraphics[width=0.85\textwidth]
	{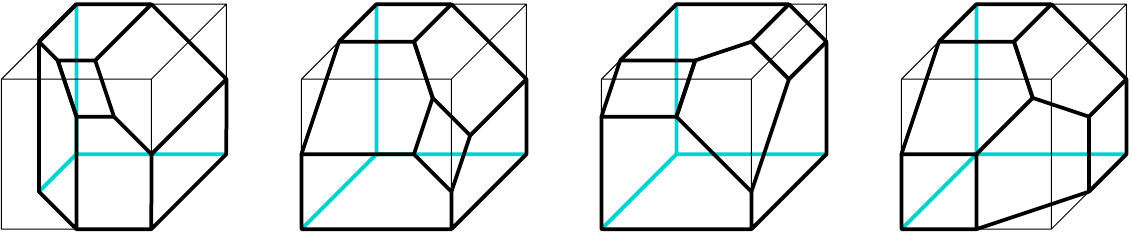}
	\caption{Four normally non-isomorphic $3$--dimensional associahedra.
	From left to right: The Loday associahedron (which is a special case of the Hohlweg--Lange associahedra), 
	the Chapoton--Fomin--Zelevinsky associa\-hedron (a special case of both Hohlweg--Lange and Santos)
	and the other two  Santos associahedra. Since they all have three pairs of parallel facets, we used a linear transformation to draw them fitting in the same cube.}
	\label{fig:ManyAssociahedra}
\end{figure}

We discuss these two generalizations in Sections~\ref{sec:exponentially-many} and~\ref{sec:triang-many}.
The construction by Santos appears in print for the first time in this paper,  so we prove things in detail. For the Hohlweg--Lange realizations we rely on the original papers for most of the details.

Let us explain what we exactly mean by \emph{different} (see more details in Section~\ref{sec:preliminaries}). Since the associahedron is simple, its realizations form an open subset in the space of $\frac{(n+3)n}{2}$-tuples of half-spaces in $\R^n$. Hence, classifying them by affine or projective equivalence does not seem the right thing to do. But for the Hohlweg--Lange and Santos constructions, with normals in $\{-1,0,1\}^n$, the set of possible normal fans obtained is finite. This suggests that one natural classification is by \emph{linear isomorphism of normal fans} or, as we call it, \emph{normal isomorphism}.
In this language:
\begin{compactitem}[$\circ$]
\item the (normal isomorphism classes of) Hohlweg--Lange associahedra are in bijection with the sequences in $\{+, -\}^{n-1}$, modulo reflection and reversal (Theorem~\ref{thm:classification_HL}, see also \cite[Cor.~2.6]{BergeronHLT}).
\item the (normal isomorphism classes of) Santos associahedra are in bijection with the triangulations of the $(n+3)$-gon, modulo dihedral symmetries of the polygon (Corollary~\ref{corollary:triangmany_distinct}).
\end{compactitem}

\noindent
The numbers of distinct associahedra obtained by the two constructions are, thus, roughly $2^{n-3}$ and  $\frac1{2(n+3)}C_{n+1}\approx  {2^{2n+1}}/{\sqrt{\pi n^5}}$; exact counts are in Sections \ref{sec:exponentially-many} and \ref{sec:triang-many}, see also Table~\ref{table:two-types}.

Although a classification of the Hohlweg--Lange associahedra appears already in~\cite{BergeronHLT}, we think our new 
derivation it has some novelty. On the one hand, the classification in~\cite{BergeronHLT} is only up to 
isometry of linear fans; it left the door open for two associahedra classified as different still being equivalent if 
a linear transformation of the normal fan is allowed (see Remark~\ref{rem:gen-associahedra-class}). On the other hand, we show that two Hohlweg--Lange associahedra coming from non-equivalent sign sequences can be distinguished by their pairs of parallel facets (see the proof of Theorem~\ref{thm:classification_HL}). 
The fact that Hohlweg--Lange asociahedra have parallel facets is obvious from the definitions, but was not used 
in~\cite{BergeronHLT}. 

The same works for Santos asociahedra: if two of them are produced by  non-equivalent triangulations, then they are not normally isomorphic (Lemma~\ref{lemma:triangmany_distinct}).
Even more so, the only Hohlweg--Lange associahedron with the same pairs of parallel facets as a Santos associahedron is the Chapoton--Fomin--Zelevinsky associahedron. That is to say (Theorem~\ref{theo:almost-disjoint}):

\begin{theorem}
\label{thm:main}
The Hohlweg--Lange and Santos  families of associahedra are almost disjoint, the only common element being the Chapoton--Fomin--Zelevinsky associahedron.
\end{theorem}

The secondary polytope construction of the associahedron has a completely different flavor,
since a continuous deformation of the polygon produces a continuous deformation of the associahedron obtained and of its normal fan. 
Moreover, the secondary polytope of a convex polygon never has parallel facets (Proposition~\ref{prop:par_cont^I}, already noticed in {\cite[Sec.~5.3]{RoSaSt03}}).
This difference  is apparent comparing Figures~\ref{fig:hex-secondary} and~\ref{fig:ManyAssociahedra}. 

In Section~\ref{sec:combinatorial_c-cluster_complex} we relate the Santos construction with the $c$-cluster complexes and the denominator fans in cluster algebras of type $A$. The $c$-cluster complexes are simplicial complexes defined by Reading in~\cite{reading_clusters_2007} following ideas 
from~\cite{MRZ03}. We obtain a simple combinatorial 
description of these complexes for Coxeter groups of type $A$, and show that they are the normal fans of some of the Santos associahedra (Proposition~\ref{prop:c-cluster_fan}). In the general case, the normal fans of the Santos associahedra can be interpreted as the denominator fans in cluster algebras of type $A$~(Proposition~\ref{prop:denominator_fan}). 
This connection suggests a natural generalization of the Santos construction to a construction of generalized cluster-associahedra in arbitrary finite Coxeter groups (Question~\ref{question:CoxeterCatalanAssociahedra}).

Let us remark that, even if both the Hohlweg--Lange and the Santos constructions have very natural interpretations (and generalizations, modulo the question above) in the context of finite Coxeter groups, they have a significant difference; their normal 
fans lie in the \emph{root space} and the \emph{weight space} respectively. This is a bit hidden in Figure~\ref{fig:ManyAssociahedra}, where we have performed a linear transformation to draw the polytopes inscribed in the same cube.

One of the questions that remains is whether there is a common generalization of the Hohlweg--Lange and the Santos construction, which may perhaps produce even more examples of ``combinatorial'' associahedra. An exhaustive search produces,  besides the four $3$-associahedra of Figure~\ref{fig:ManyAssociahedra}, another four $3$-associahedra that arise by cutting three faces of a $3$-cube (see Figure~\ref{fig:MoreAssociahedra}). Do these admit a natural combinatorial interpretation as well?
 
\begin{figure}[ht]
	\centering
	\includegraphics
	[width=0.85\textwidth]
	{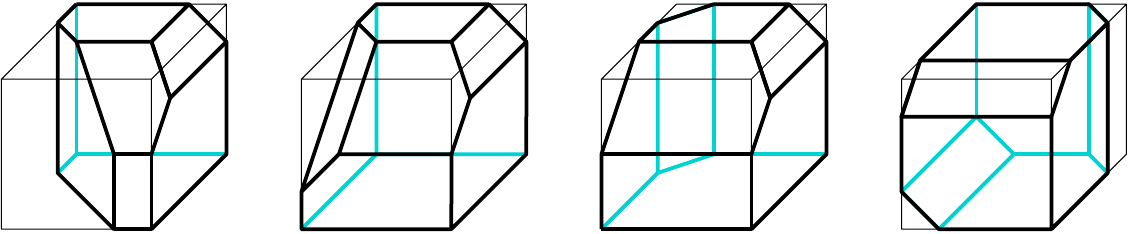}
	\caption{More $3$-associahedra inscribed in a $3$-cube. The $3$-associahedron is the only  simple $3$-polytope with nine facets all of which are quadrilaterals or pentagons.}
	\label{fig:MoreAssociahedra}
\end{figure}

As a final remark, part of the motivation of this paper was to try to find out \emph{what is the most ``natural'' or ``canonical'' realization of the associahedron}. The answer is not clear. If one wants to realize all the combinatorial symmetries  of the polytope (which, as we show in Lemma~\ref{lemma:automorphism_rotation-reflection}, are precisely the dihedral symmetries of the $(n+3)$-gon) then the best candidate is the secondary polytope of a regular polygon (Figure~\ref{fig:hex-secondary}). But if small integer coordinates for vertices and facet normals are seeked then you certainly want one of the Hohlweg--Lange or Santos associahedra (Figure~\ref{fig:ManyAssociahedra}). 
Among them, the Loday associahedron sticks out as the most ubiquitous in the literature~\cite{Bu08,HoLa07,Lee89,Lo04,PiSa11,Po05,RoSaSt03,ShniderSternberg93},
but the Chapoton--Fomin--Zelevinsky associahedron is the only one produced both by the Hohlweg--Lange and the Santos constructions.

\section{Two preliminaries}\label{sec:preliminaries}

Let $P_{n+3}$ be a convex $(n+3)$-gon.
An \emph{associahedron} $\mathrm{Ass}_n$ is an $n$-dimensional simple polytope whose 
poset of non-empty faces is isomorphic to the poset of non-crossing sets of diagonals 
of $P_{n+3}$, ordered by reverse inclusion. %
In particular, the vertices of the associahedron correspond to the \emph{triangulations}Ê~of~$P_{n+3}$ and its facets to the internal diagonals.%

\subsubsection{Normal isomorphism}

The goal of this paper is to compare different types 
of constructions of the associahedron, saying which 
ones produce equivalent polytopes, in a suitable sense.   
The following notion reflects the fact that the main constructions that we are going to discuss produce associahedra whose normal vectors have small integer coordinates, usually $0$ or $\pm1$. In these constructions the normal fan of the associahedron can be considered canonical, while there is still freedom in the right-hand sides of the inequalities. 
Recall that the normal fan of a polytope $P\subset \R^n$ is the partition of $(\R^*)^n$ into the normal cones of the different faces of $P$. Each $1$-dimensional cone (ray) in the normal fan is generated by the (exterior) normal vector to a facet of $P$ and the $n$-dimensional cones are normal to the vertices of $P$. Since all the polytopes in this paper are simple, their normal fans are \emph{simplicial}: every cone is generated by an independent set of vectors.
(See \cite[Sec.~7.1]{Zi} for further discussion of fans and of normal fans.) 

This leads us to use the following notion of equivalence.

\begin{definition}
\label{def:nl-equivalent} Two complete fans in real vector spaces $V$ and $V'$ of the same dimension are \emph{linearly isomorphic} if there is a linear isomorphism $V\to V'$ sending each cone of one to a cone of the other. Two polytopes $P$ and $P'$ are \emph{normally isomorphic}Ê~if they have linearly isomorphic normal fans.
\end{definition} 

Normal isomorphism is weaker than the usual notion of \emph{normal equivalence}, in which the two polytopes $P$ and $P'$ are assumed embedded \emph{in the same space} and their normal fans are required to be exactly the same, not only linearly isomorphic. One easy way to check that two polytopes are \emph{not} normally isomorphic is to show that no combinatorial isomorphism sends parallel facets of one to parallel facets of the other.

\subsubsection{Automorphisms of the associahedron}

The following lemma is very useful in order to 
disprove that two associahedra are normally isomorphic. 
It implies that all normal (or combinatorial, for that matter) isomorphisms between associahedra come from
isomorphisms between the $(n+3)$-gons defining them.

\begin{lemma}\label{lemma:automorphism_rotation-reflection}
All  automorphisms of the face lattice of the
associahedron $\mathrm{Ass}_n$ are induced
by symmetries of the $(n+3)$-gon. In particular, for $n\ge 2$,
the automorphism group of the face lattice of the
associahedron $\mathrm{Ass}_n$ is isomorphic to the dihedral group of order $2n+6$.
\end{lemma} 

\begin{proof}
The second sentence follows from the first one since the only symmetry of the $k$-gon sending every diagonal to itself is the identity, for $k\ge 5$.
For the first sentence, suppose $\varphi$ is an automorphism of the face lattice of the associahedron $\mathrm{Ass}_n$,
and let $D$ be the set of all diagonals of a convex $(n+3)$-gon. 
$\varphi$ induces a natural bijection
\[
\widetilde \varphi : D \longrightarrow D
\] 
such that for any two diagonals $\delta, \delta' \in D$ we have:
\[
\delta \text{ crosses } \delta' \ \ 
\Longleftrightarrow \ \ 
\widetilde \varphi (\delta) \text{ crosses } \widetilde \varphi (\delta').
\] 

Call \emph{length} of a diagonal $\delta \in D$
the minimum between the lengths of 
the two paths that connect the 
two end points of $\delta$ 
on the boundary of the $(n+3)$-gon. 
Since the length of $\delta$ 
is determined by the number of diagonals 
that cross $\delta$,
\[
\text{length}(\delta)=\text{length}(\widetilde \varphi (\delta) ).
\] 

Now, two diagonals of the $(n+3)$-gon have a common vertex if and only if there is a diagonal of length $2$ crossing both of them. In particular, $\widetilde \varphi$ sends diagonals with a common vertex to diagonals with common vertex and can thus be understood as a graph automorphism of $D$, when $D$ is regarded as the graph $K_{n+3}\setminus C_{n+3}$ (the complete graph minus the Hamiltonian cycle along the boundary of the $(n+3)$-gon). The only such automorphisms are clearly the dihedral automorphisms of the cycle $C_{n+3}$.
\end{proof}

\section{Three constructions of the associahedron}\label{review}

In this section we review three very nice constructions in geometric combinatorics that have the associahedron as particular cases.

\subsection{The associahedron as a secondary polytope}
\label{section_secondarypolytope}

The secondary polytope is an ingenious construction motivated by the theory of hypergeometric functions as developed by Gelfand, Kapranov and Zelevinsky~\cite{GKZ94}, later generalized and explained in terms of fiber polytopes by Billera and Sturmfels~\cite{BS92}. In this section we recall the basic definitions and main results related to this topic, which yield in particular that the secondary polytope of any convex $(n+3)$-gon is an $n$-dimensional associahedron. For more detailed presentations we refer to~\cite[Ch. 5]{LoRaSa10}
and~\cite[Lect. 9]{Zi}.  All the subdivisions and triangulations of polytopes that appear in the following
are understood to be without new vertices.

\subsubsection{The secondary polytope construction} 

\begin{definition}[GKZ vector/secondary polytope] \label{def_GKZ}
Let $Q$ be a $d$-dimensional convex polytope with $n+d+1$ vertices. The \emph{GKZ vector} 
$v(t)\in \R^{n+d+1}$ of a triangulation $t$ of $Q$ is 
\begin{eqnarray*}
v(t)\ := \ \sum_{i=1}^{n+d+1} \text{vol} (\text{star}_t(i)) e_i 
\ \ =\ \ \sum_{i=1}^{n+d+1} \sum_{\sigma\in t\,:\, i\in\sigma}    \text{vol} (\sigma) e_i
\end{eqnarray*} 
The \emph{secondary polytope} of $Q$ is defined as
\[
\Sigma (Q)\ :=\ \text{conv}\{ v(t) : t \text{ is a triangulation of } Q \}.
\]
\end{definition}

\begin{theorem} [Gelfand--Kapranov--Zelevinsky~\cite{GZK90}] \label{theoGKZ}
Let $Q$ be a $d$-dimensional convex polytope with $m=n+d+1$ vertices. 
The secondary polytope $\Sigma (Q)$ has the following properties:
\begin{compactenum}[\rm(i)]
\item $\Sigma (Q)$ is an $n$-dimensional polytope.
\item The vertices of $\Sigma (Q)$ are in bijection with the regular triangulations of~$Q$.
\item The faces of $\Sigma (Q)$ are in bijection with the regular subdivisions of~$Q$.
\item The face lattice of $\Sigma (Q)$ is isomorphic to the lattice of regular subdivisions of $Q$, 
ordered by refinement.
\end{compactenum}
\end{theorem}

\subsubsection{The associahedron as the secondary polytope of a convex $(n+3)$-gon}

\begin{definition} The \emph{Gelfand--Kapranov--Zelevinsky associahedron} $\GKZ(Q)\subset \R^{n+3}$ is defined as
the ($n$-dimensional) secondary polytope of a convex $(n+3)$-gon $Q\subset\R^2$:
\[
\GKZ(Q):=\Sigma (Q).
\]
\end{definition}

\begin{corollary} [Gelfand--Kapranov--Zelevinsky~\cite{GZK90}]
$\GKZ(Q)$ is an $n$-dimensional associahedron.
\end{corollary}

There is one feature that distinguishes the associahedron as a secondary polytope from all the other constructions that we mention in this paper: the absence of parallel facets. 

\begin{proposition}[Rote--Santos--Streinu {\cite[Sec.~5.3]{RoSaSt03}}]
\label{prop:par_cont^I}
Let $Q$ be a convex $(n+3)$-gon. Then
$\GKZ(Q)$ has no parallel facets for $n\geq 2$.
\end{proposition}

This was stated without proof by Rote, Santos and Streinu~\cite[Sec.~5.3]{RoSaSt03}. Here we offer a proof,
based on the understanding of the facet normals in secondary polytopes.
Let $Q$ be an arbitrary $d$-polytope with $n+d+1$ vertices $\{q_1,\dots,q_{n+d+1}\}$, so that $\GKZ(Q)$ lives in $\R^{n+d+1}$, although it has dimension $n$. In the theory of secondary polytopes one thinks of each linear functional $\R^{n+d+1}\to \R$  as a function $\omega:\operatorname{vertices}(Q)\to \R$ assigning a value $\omega(q_i)$ to each vertex $q_i$. In turn, to each triangulation $t$ of $Q$ (with no additional vertices) and any such $\omega$ one associates the function $g_{\omega,t}:Q \to \R$ which takes the value $\omega(q_i)$ at each $q_i$ and is affine linear on each simplex of $t$. That is, we use $t$ to piecewise linearly interpolate a function whose values $(\omega(q_1),\dots,\omega(q_n))$ we know on the vertices of~$Q$. The main result we need is the following equality
for every $\omega$ and every triangulation $t$  (see, e.g.,~\cite[Thm. 5.2.16]{LoRaSa10}):
\[
\langle \omega, v(t)\rangle =(d+1) \int_Q  g_{\omega,t}(x) dx.
\]

\noindent In particular:
\begin{compactitem}[$\circ$]
\item If $\omega$ is affine-linear (that is, if the points $\{(q_1,\omega_1),\dots, (q_{n+d+1},\omega_{n+d+1})\}\subset \R^{n+d+1}\times \R$ lie in a hyperplane) then  
$\langle \omega, v(t)\rangle$  is the same for all $t$. Moreover, the converse is also true: The affine-linear $\omega$'s form the lineality space of the normal fan of~$\GKZ(Q)$.

\item An $\omega$ lies in the linear cone of the (inner) normal fan of $\GKZ(Q)$ corresponding to a certain triangulation $t$ (that is, $\langle \omega, v(t)\rangle
\le \langle \omega, v(t')\rangle$ for every other triangulation $t'$) if and only if the function $g_{\omega,t}$ is convex; that is to say, if its graph is a convex hypersurface.
\end{compactitem}

\begin{proof} [Proof of Proposition~\ref{prop:par_cont^I}]
With the previous description in mind we can identify the facet normals of the secondary polytope of a polygon $Q$. For this we use the correspondence: 
\[
\begin{array}{rcl}
 \text{vertices}  & \longleftrightarrow & \text{triangulations of } Q \\
 \text{facets}  & \longleftrightarrow & \text{diagonals of } Q  \\
\end{array}
\]
For a given diagonal $\delta$ of~$Q$, denote by $F_\delta$ the facet of $\GKZ(Q)$ corresponding  to $\delta$. The vector normal to $F_\delta$ is not unique, since adding to any vector normal to $F_\delta$ an affine-linear $\omega_0$ we get another one. 
One natural choice is  
\[
\omega_\delta(q_i):=\operatorname{dist}(q_i,l_\delta),
\]
where $l_\delta$ is the line containing $\delta$ and $\operatorname{dist}(\cdot,\cdot)$ is the Euclidean distance.
Indeed, $\omega_\delta$ lifts the vertices of $Q$ on the same side of $\delta$ to lie in a half-plane in $\R^3$, with both half-planes having $\delta$ as their common intersection. That is, $g_{\omega_\delta,t}$ is convex for every $t$ that uses $\delta$. But another choice of normal vector is better for our purposes: choose one side of $l_\delta$ to be called positive and take
\[
\omega^+_\delta(q_i):=\begin{cases}\operatorname{dist}(q_i,l_\delta) &\text{if } q_i\in l_\delta^+\\ 0 & \text{if }   q_i\in l_\delta^-\end{cases}.
\]
For the end-points of $\delta$, which lie in both $l_\delta^+$ and $l_\delta^-$, there is no ambiguity since both definitions give the value $0$. Again, $\omega^+_\delta$ is a normal vector to $F_\delta$ since it lifts points on either side of $l_\delta$ to lie in a plane.

We are now ready to prove the theorem.
If two diagonals $\delta$ and $\delta'$ of $Q$ do not cross, then they can simultaneously be used in a triangulation. Hence, the corresponding facets $F_\delta$ and $F_{\delta'}$ meet, and they cannot be parallel. So, assume in what follows that $\delta$ and $\delta'$ are two crossing diagonals. Let $\delta=pr$ and $\delta'=qs$, with $pqrs$ being cyclically ordered along $Q$. Since $n\geq 2$ there is at least another vertex $a$ in $Q$. Without loss of generality suppose $a$ lies between $s$ and $p$. Now, we call negative the side of $l_\delta$ and the side of $l_{\delta'}$ containing $a$, and consider the normal vectors $\omega^+_\delta$ and $\omega^+_{\delta'}$ as defined above. They take the following values on the five points of interest:
\begin{eqnarray*}
\omega^+_\delta(a)=0,\quad\omega^+_\delta(p)=0,\quad\omega^+_\delta(q)>0,\quad\omega^+_\delta(r)=0,\quad\omega^+_\delta(s)=0,\\
\omega^+_{\delta'}(a)=0,\quad\omega^+_{\delta'}(p)=0,\quad\omega^+_{\delta'}(q)=0,\quad\omega^+_{\delta'}(r)>0,\quad\omega^+_{\delta'}(s)=0.
\end{eqnarray*}

Suppose that $F_\delta$ and $F_{\delta'}$ were parallel. This would imply that $\delta$ and $\delta'$ are linearly dependent or, more precisely, that there is a linear combination of them that gives an affine-linear $\omega$ (in the lineality space of the normal fan). But any (non-trivial) linear combination 
$\omega$ of $\omega^+_{\delta}$ and $\omega^+_{\delta'}$ necessarily takes the following values on our five points,
which implies that $\omega$ is not affine-linear:
\begin{eqnarray*}
\omega(a)=0,\quad\omega(p)=0,\quad\omega(q)\ne0,\quad\omega(r)\ne0,\quad\omega(s)=0.\\[-12mm]{}
\end{eqnarray*}
\end{proof}

\begin{remark}
\label{rem:weakly_convex}
The secondary polytope of  
points $\{q_1,\dots,q_{n+3}\}$ in the plane that are not  
the vertices of a convex polygon is, in general, not an associahedron.
But there is a case in which it is: 
if the points are  placed on the boundary of an $m$-gon (with 
$m\le n+3$) in such a way that no four of them lie on the same edge. 
By the arguments in the proof above, a necessary condition for
the associahedron obtained to have parallel facets is that $m\le 4$.
For $m=4$ we can obtain associahedra up to dimension 4 with
exactly one pair of parallel facets (those corresponding to the
main diagonals of the quadrilateral). For $m=3$, we can obtain 
2-dimensional associahedra with two pairs of parallel facets, and 
3-dimensional associahedra with three pairs of parallel facets. The latter is obtained
for six 
points $\{p,q,r,a,b,c\}$ with $p$, $q$ and $r$ being the vertices of 
a triangle and $a\in pq$, $b\in qr$ and $c\in ps$ intermediate 
points in the three sides. The associahedron obtained has the following three 
pairs of parallel facets:
\[
F_{pq} || F_{ar},\quad F_{qr} || F_{bs},\quad F_{ps} || F_{cq}.
\]
It is normally isomorphic to the right-most associahedron of Figure~\ref{fig:ManyAssociahedra}.
\end{remark}

\begin{remark}
\label{rem:pt-polytope}
Rote, Santos and Streinu~\cite{RoSaSt03} introduce a \emph{polytope of pseudo-triangulations} associated to each finite set $A$ of $m$ points (in general position) in the plane. This polytope lives in $\R^{2m}$ and has dimension $m+3+i$, where $i$ is the number of points interior to $\conv(A)$. They show that for points in convex position their polytope is affinely isomorphic to the secondary polytope for the same point set. Their constructions uses rigidity theoretic ideas: the edge-direction joining two neighboring triangulations $t$ and $t'$ is the vector of velocities of the (unique, modulo translation and rotation) infinitesimal flex of the embedded graph of $t\cap t'$.
\end{remark}

\subsection{The associahedron as a generalized permutahedron}

We here review two  constructions of the associahedron: one by Postnikov~\cite{Po05} and 
one by Rote--Santos--Streinu~\cite{RoSaSt03} (different from the one in Remark~\ref{rem:pt-polytope}).
The main goal of this section is to prove that these two constructions produce 
affinely equivalent results. In both constructions only the normal fan is fixed. Equivalently, there is freedom in the construction for the right-hand sides of facet-defining inequalities, and the space of valid right-hand sides is explicitly described. Specific right-hand sides produce, respectively, 
the realizations by Loday~\cite{Lo04} and Buchstaber~\cite{Bu08}, which turn out to be
affinely equivalent as well.   

\subsubsection{The Postnikov associahedron}

The Postnikov associahedron is a  special case of the family of generalized permutahedra studied in~\cite{Po05}. 
Recall that the standard $n$-dimensional permutahedron is the polytope
\begin{equation}
\begin{array}{rl}
\big\{  (x_1,\dots x_{n+1}) \in \R^{n+1} : 
& \sum\limits_{i\in S} x_i \geq \binom{| S | + 1)}{2}   \text{ for all } S\subsetneq[n+1], \quad \sum\limits_{i\in [n+1]} x_i=\binom{n+2}2
\big\}.
\end{array}
\label{eq:permutahedron}
\end{equation}
Equivalently, it equals 
the convex hull of the $n!$ points in $\R^{n+1}$ obtained by permuting coordinates in $(1,\dots,n+1)$ and also 
the Minkowski sum of the edges of the standard simplex $\{  (x_1,\dots, x_{n+1}) \in \R^{n+1} : 
 \sum x_i =1,\, x_i\ge 0\}$.
A generalized permutahedron is a polytope with facet normals contained in those of the standard permutahedron and such that the collection of right hand side parameters of the defining inequalities belongs to (the closure of) the deformation cone of the standard permutahedron. Besides associahedra, generalized permutahedra include many interesting polytopes such as cyclohedra, graph associahedra, nestohedra, and all Minkowski sums of dilated faces of a standard simplex.  

It follows from the Minkowski sum description of the permutahedron that every positively weighted Minkowski sum of arbitrary faces of the standard simplex is a generalized permutahedron (the converse is only partially true; see Remark~\ref{rem:minkowski-sum}). Following Postnikov~\cite{Po05} we use this to define the Loday and Postnikov associahedra:

\begin{definition}
For any vector $\mathbf{a}=\{\mathrm{a}_{ij}>0: 
1\leq i\leq j \leq n+1\}$ of positive parameters we call 
\emph{Postnikov associahedron} the polytope
\[
\Post(\mathbf{a}):=\sum_{1\leq i\leq j\leq n+1}\mathrm{a}_{ij}\Delta_{[i,\dots ,j]},
\]
where $\Delta_{[i,\dots ,j]}$ denotes the simplex conv$\{e_i,e_{i+1},\dots , e_j\}$ in $\R^{n+1}$. 
The special case where $\mathrm{a}_{ij}=1$ for all $i,j$ is the \emph{Loday associahedron}.
\end{definition}

\begin{proposition}[Postnikov {\cite[Sec.~8.2]{Po05}}] 
$\Post(\mathbf{a})$ is an $n$-dimensional associahedron. 
\end{proposition}

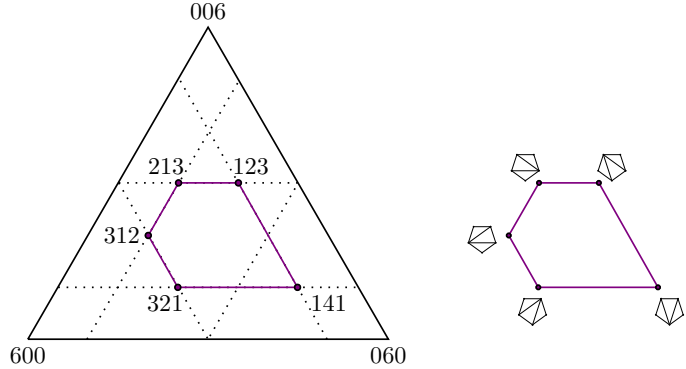
\begin{figure}[ht]
\begin{centering}
\resizebox{!}{50mm}{
\begin{tikzpicture}[scale = 1.5]

\draw[line width = 1pt] ( 0,0) -- (5,0) -- (2.5,4.33) -- (0,0);
\draw (0,0)node[scale = 1.3, anchor =north ]{600};
\draw (5,0)node[scale = 1.3, anchor =north ]{060};
\draw (2.5,4.33)node[scale = 1.3, anchor =south ]{006};
\draw[line width = 1pt, loosely dotted] (0.42,0.72) -- (4.58,0.72) ;
\draw[line width = 1pt, loosely dotted] (1.26,2.17) -- (2.5,0);
\draw[line width = 1pt, loosely dotted] (1.26,2.17) -- (3.75,2.17)--(2.5,0);
\draw[line width = 1pt, loosely dotted] (0.82,0) -- (2.93,3.59);
\draw[line width = 1pt, loosely dotted] (4.15,0) -- (2.09,3.63);

\draw[line width = 1pt, color = violet] (2.08, 0.72) -- (3.74,0.72) -- (2.92,2.17) -- (2.09,2.17) -- (1.67,1.44) -- (2.08,0.72); 

\filldraw[color = violet] (2.08,0.72) circle (0.04)node[scale = 1.3, anchor =north, color = black, xshift = -0.2cm ]{321};
\draw[line width = 1pt] (2.08,0.72) circle (0.04);

\filldraw[color = violet] (3.74,0.72) circle (0.04)node[scale = 1.3, anchor =north, color = black, xshift = 0.5cm ]{141};
\draw[line width = 1pt] (3.74,0.72) circle (0.04);

\filldraw[color = violet] (2.92,2.17) circle (0.04)node[scale = 1.3, anchor =south, color = black, xshift = 0.2cm ]{123};
\draw[line width = 1pt] (2.92,2.17) circle (0.04);

\filldraw[color = violet] (2.09,2.17)  circle (0.04)node[scale = 1.3, anchor =south, color = black, xshift =- 0.2cm ]{213};
\draw[line width = 1pt] (2.09,2.17) circle (0.04);

\filldraw[color = violet] (1.67,1.44)  circle (0.04)node[scale = 1.3, anchor =east, color = black]{312};
\draw[line width = 1pt] (1.67,1.44) circle (0.04);

\newdimen\L
\L=0.2cm
\newcommand{\smallpentagon}[4]{

\draw[xshift=#1, yshift = #2,rotate =#3, ] (0:\L)
\foreach \x in {72,144,...,360} {
      -- (\x:\L)
 } -- cycle (90:\L)  ;

\foreach \x in {0,72,144,...,360} {
     \draw[fill,xshift=#1, yshift = #2,rotate =#3] (\x:\L) circle (0.005);
     
 } ;


\ifcase #4

\or

\draw[xshift=#1, yshift = #2,rotate =#3] (0:\L) -- (144:\L);

\or

\draw[,xshift=#1, yshift = #2,rotate =#3] (0:\L) -- (144:\L);
\draw[,xshift=#1, yshift = #2,rotate =#3] (144:\L) -- (288:\L);

\fi

}

\draw[line width = 1pt, color = violet, xshift = 5cm] (2.08, 0.72) -- (3.74,0.72) -- (2.92,2.17) -- (2.09,2.17) -- (1.67,1.44) -- (2.08,0.72); 

\filldraw[color = violet] (7.08,0.72) circle (0.03);
\draw[line width = 1pt] (7.08,0.72) circle (0.03);

\filldraw[color = violet] (8.74,0.72) circle (0.03);
\draw[line width = 1pt] (8.74,0.72) circle (0.03);

\filldraw[color = violet] (7.92,2.17) circle (0.03);
\draw[line width = 1pt] (7.92,2.17) circle (0.03);

\filldraw[color = violet] (7.09,2.17)  circle (0.03);
\draw[line width = 1pt] (7.09,2.17) circle (0.03);

\filldraw[color = violet] (6.67,1.44)  circle (0.03);
\draw[line width = 1pt] (6.67,1.44) circle (0.03);

\smallpentagon{7cm}{0.4cm}{-90}{2};
\smallpentagon{8.9cm}{0.4cm}{-234}{2};
\smallpentagon{8.1cm}{2.4cm}{-18}{2};
\smallpentagon{6.9cm}{2.4cm}{-162}{2};
\smallpentagon{6.3cm}{1.4cm}{54}{2};
\end{tikzpicture}
}
\end{centering}
\caption{The Loday associahedron $\Post (\mathbf 1)$ with the coordinates of its vertices.}
\label{fig:loday}
\end{figure}

Using a special case of~\cite[Prop.~6.3]{Po05}, the Postnikov associahedron can be described in terms of inequalities as follows.

\begin{lemma}[Postnikov~\cite{Po05}] 
\label{lemma:post-inequalities}
\begin{eqnarray*}
\Post (\mathbf a) = \{(x_1,\dots,x_{n+1})\in \R^{n+1}: 
\sum_{p < i < q} x_i \geq f_{p,q} \hspace{3mm}
\text{ for } 0 \leq p < q \leq n+2,
\\ x_1+\cdots+x_{n+1}=f_{0,n+2}\},
\end{eqnarray*}
where $f_{p,q}=\sum_{p<i\leq j<q} \mathrm a_{i,j} $.
\end{lemma}

Conversely, the Minkowski weights $\mathrm a_{i,j}$ of a Postnikov associahedron defined by right-hand sides $f_{p,q}$ can be obtained by M\"obius inversion. This is thoroughly analyzed in~\cite{Lange11}.

The facet of $\Post (\mathbf a)$ labeled by a pair $({p,q})$
corresponds to the diagonal $pq$
of an $(n+3)$-gon with vertices labeled 
in counterclockwise direction
from 0 to $n+2$. 
It is obvious from the description in Lemma~\ref{lemma:post-inequalities} that
$\Post(\mathbf{a})$ has exactly $n$ pairs of parallel facets. These correspond to the pairs of diagonals 
$( \{ 0,i+1 \}, \{ i,n+2 \})$ for $1\leq i \leq n$, as illustrated in 
{\normalfont Figure~\ref{Parallel_ass^III}}.
This is a particular case of Proposition~\ref{prop:parallel_typeI}.

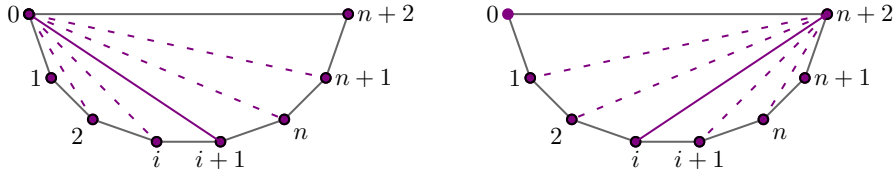
\begin{figure}[ht]
\begin{centering}
\resizebox{!}{24mm}{

\begin{tikzpicture}[scale = 0.5]

\draw[line width = 0.9pt,color = black!60](-5,4)--(-4.3,2)--(-3,0.7)-- (-1,0) -- (1,0) -- (3,0.7)--(4.3,2)--(5,4) -- (-5,4) ;
\filldraw[line width =0.9pt,color = violet] (-1,0) circle (0.15)node[anchor = north, color = black, scale = 1] {$i$};;
\draw[line width =0.9pt] (-1,0) circle (0.15);
\filldraw[line width =0.9pt,color = violet] (1,0) circle (0.15) node[anchor = north, color = black, scale = 1] {$i+1$};;
\draw[line width =0.9pt] (1,0) circle (0.15);
\filldraw[line width =0.9pt,color = violet] (3,0.7) circle (0.15)node[anchor =north  west, color = black, scale = 1] {$n$};;
\draw[line width =0.9pt] (3,0.7) circle (0.15);
\filldraw[line width =0.9pt,color = violet] (4.3,2) circle (0.15)node[anchor = west, color = black, scale = 1] {$n+1$};;
\draw[line width =0.9pt] (4.3,2) circle (0.15);
\filldraw[line width =0.9pt,color = violet] (5,4) circle (0.15)node[anchor = west, color = black, scale = 1] {$n+2$};
\draw[line width =0.9pt] (5,4) circle (0.15);
\filldraw[line width =0.9pt,color = violet] (-3,0.7) circle (0.15)node[anchor = north  east, color = black, scale = 1] {$2$};
\draw[line width =0.9pt] (-3,0.7) circle (0.15);
\filldraw[line width =0.9pt,color = violet] (-4.3,2) circle (0.15)node[anchor =  east, color = black, scale = 1] {$1$};
\draw[line width =0.9pt] (-4.3,2) circle (0.15);
\filldraw[line width =0.9pt,color = violet] (-5,4) circle (0.15) node[left, color = black, scale = 1] {$0$};
\draw[line width =0.9pt] (-5,4) circle (0.15);


\draw[line width = 0.9pt,loosely dashed, color=violet](-5,4) -- (4.3,2);
\draw[line width = 0.9pt,loosely dashed, color=violet](-5,4) -- (-3,0.7);
\draw[line width = 0.9pt,loosely dashed, color=violet](-5,4) -- (-1,0);
\draw[line width = 0.9pt, color=violet](-5,4) -- (1,0);
\draw[line width = 0.9pt,loosely dashed, color=violet](-5,4) -- (3,0.7);

\draw[line width = 0.9pt,color = black!60, xshift = 15cm](-5,4)--(-4.3,2)--(-3,0.7)-- (-1,0) -- (1,0) -- (3,0.7)--(4.3,2)--(5,4) -- (-5,4) ;
\filldraw[line width =0.9pt,color = violet,xshift = 15cm] (-1,0) circle (0.15)node[anchor = north, color = black, scale = 1] {$i$};;
\draw[line width =0.9pt,xshift = 15cm] (-1,0) circle (0.15);
\filldraw[line width =0.9pt,color = violet,xshift = 15cm] (1,0) circle (0.15) node[anchor = north, color = black, scale = 1] {$i+1$};;
\draw[line width =0.9pt,xshift = 15cm] (1,0) circle (0.15);
\filldraw[line width =0.9pt,color = violet,xshift = 15cm] (3,0.7) circle (0.15)node[anchor =north  west, color = black, scale = 1] {$n$};;
\draw[line width =0.9pt,xshift = 15cm] (3,0.7) circle (0.15);
\filldraw[line width =0.9pt,color = violet,xshift = 15cm] (4.3,2) circle (0.15)node[anchor = west, color = black, scale = 1] {$n+1$};;
\draw[line width =0.9pt,xshift = 15cm] (4.3,2) circle (0.15);
\filldraw[line width =0.9pt,color = violet,xshift = 15cm] (5,4) circle (0.15)node[anchor = west, color = black, scale = 1] {$n+2$};
\draw[line width =0.9pt,xshift = 15cm] (5,4) circle (0.15);
\filldraw[line width =0.9pt,color = violet,xshift = 15cm] (-3,0.7) circle (0.15)node[anchor = north  east, color = black, scale = 1] {$2$};
\draw[line width =0.9pt,xshift = 15cm] (-3,0.7) circle (0.15);
\filldraw[line width =0.9pt,color = violet,xshift = 15cm] (-4.3,2) circle (0.15)node[anchor =  east, color = black, scale = 1] {$1$};
\draw[line width =0.9pt,xshift = 15cm] (-4.3,2) circle (0.15);
\filldraw[line width =0.9pt,color = violet,xshift = 15cm] (-5,4) circle (0.15) node[left, color = black, scale = 1] {$0$};
\draw[line width =0.9pt] (-5,4) circle (0.15);


\draw[line width = 0.9pt,loosely dashed, color=violet,xshift = 15cm](5,4) -- (-4.3,2);
\draw[line width = 0.9pt,loosely dashed, color=violet,xshift = 15cm](5,4) -- (-3,0.7);
\draw[line width = 0.9pt, color=violet,xshift = 15cm](5,4) -- (-1,0);
\draw[line width = 0.9pt,loosely dashed, color=violet,xshift = 15cm](5,4) -- (1,0);
\draw[line width = 0.9pt,loosely dashed, color=violet,xshift = 15cm](5,4) -- (3,0.7);
\end{tikzpicture}

}
\end{centering}
\caption{Diagonals of the $(n+3)$-gon that correspond to the pairs of parallel facets of 
		both $\Post(\mathbf{a})$ and $\RSS(\mathbf{g})$. }
\label{Parallel_ass^III}
\end{figure}

\subsubsection{The Rote--Santos--Streinu associahedron}

By ``generalizing'' the construction of Remark~\ref{rem:pt-polytope} to sets of points along a line, 
Rote, Santos and Streinu~\cite{RoSaSt03} obtain a second realization of 
the associahedron.
 
\begin{definition}\label{def_RSSassociahedron}
The \emph{Rote--Santos--Streinu associahedron} is the polytope
\[
\RSS(\mathbf g)=\{(y_0,y_1,\dots , y_{n+1})\in \R^{n+2} : 
y_j-y_i \geq g_{i,j} \text{ for } j>i,\ 
y_0= 0,\ 
y_{n+1}=g_{0,n+1}
\},
\]
where $\mathbf g =(g_{i,j})_{0 \leq i < j \leq n+1}$ 
is any vector with real coordinates satisfying
\begin{eqnarray}
g_{i,l} + g_{j,k}  >   g_{i,k} + g_{j,l} & \text{ for all } & i<j \leq k < l, 
\label{eqn_RSS_rhs1}
 \\
\label{eqn_RSS_rhs2}
g_{i,l} > g_{i,k} + g_{k,l} & \text{ for all }  & i < k < l.
\end{eqnarray}
\end{definition}

\begin{proposition}[Rote--Santos--Streinu {~\cite[Sec.~5.3]{RoSaSt03}}]
If the vector $\mathbf g$ satisfies  inequalities {\rm (\ref{eqn_RSS_rhs1})}
 and {\rm (\ref{eqn_RSS_rhs2})}
then
$\RSS (\mathbf g)$ is an $n$-dimensional associahedron. 
\end{proposition}

A particular example of valid parameters $\mathbf g$ is 
given by $\mathbf {g_0}$: $g_{i,j}=i(i-j)$. 
In this case we get the realization of the associahedron 
introduced by Buchstaber in
\cite[Lect.~\rm {II} Sec.~5]{Bu08}.

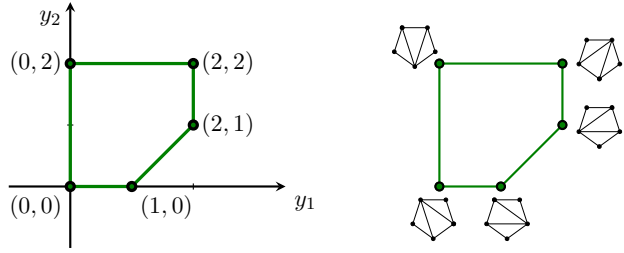
\begin{figure}[ht]
\begin{centering}
\resizebox{!}{33mm}{

\begin{tikzpicture}[scale=1];

\draw[-stealth,line width = 1pt] (0,-1) -- (0,3)  node [anchor = north east, color = black]{$y_2$};
\draw[-stealth,line width = 1pt](-1,0)--(3.5,0) node [anchor = north west, color = black]{$y_1$};
\draw[line width = 0.5pt] (2,-0.05) -- (2,0.05);
\draw[line width = 0.5pt] (-0.05,1) -- (0.05,1);

\draw[line width = 1.5pt,  color = green!50!black](0,0) -- (1,0);
\draw[line width = 1.5pt,  color = green!50!black](0,0) -- (0,2);
\draw[line width = 1.5pt,  color = green!50!black](0,2) -- (2,2);
\draw[line width = 1.5pt,  color = green!50!black](2,2) -- (2,1);
\draw[line width = 1.5pt,  color = green!50!black](2,1) -- (1,0);

\draw[fill, color = green!50!black] (0,0) circle (0.07) node [anchor = north east, color = black]{$(0,0)$};
\draw[line width = 1.5pt] (0,0) circle (0.065);
\draw[fill, color = green!50!black] (1,0) circle (0.07)node [anchor = north west, color = black]{$(1,0)$};
\draw[line width = 1.5pt] (1,0) circle (0.065);
\draw[fill, color = green!50!black] (0,2) circle (0.07)node [anchor =  east, color = black]{$(0,2)$};
\draw[line width = 1.5pt] (0,2) circle (0.065);
\draw[fill, color = green!50!black] (2,2) circle (0.07)node [anchor =  west, color = black]{$(2,2)$};
\draw[line width = 1.5pt] (2,2) circle (0.065);
\draw[fill, color = green!50!black] (2,1) circle (0.07)node [anchor =  west, color = black]{$(2,1)$};
\draw[line width = 1.5pt] (2,1) circle (0.065);

\draw[line width = 1pt,  color = green!50!black, xshift = 6cm](0,0) -- (1,0);
\draw[line width = 1pt,  color = green!50!black, xshift = 6cm](0,0) -- (0,2);
\draw[line width = 1pt,  color = green!50!black, xshift = 6cm](0,2) -- (2,2);
\draw[line width = 1pt,  color = green!50!black, xshift = 6cm](2,2) -- (2,1);
\draw[line width = 1pt,  color = green!50!black, xshift = 6cm](2,1) -- (1,0);

\draw[fill, color = green!50!black, xshift = 6cm] (0,0) circle (0.07);
\draw[line width = 1pt, xshift = 6cm] (0,0) circle (0.065);
\draw[fill, color = green!50!black, xshift = 6cm] (1,0) circle (0.07);
\draw[line width = 1pt, xshift = 6cm] (1,0) circle (0.065);
\draw[fill, color = green!50!black, xshift = 6cm] (0,2) circle (0.07);
\draw[line width = 1pt, xshift = 6cm] (0,2) circle (0.065);
\draw[fill, color = green!50!black, xshift = 6cm] (2,2) circle (0.07);
\draw[line width = 1pt, xshift = 6cm] (2,2) circle (0.065);
\draw[fill, color = green!50!black, xshift = 6cm] (2,1) circle (0.07);
\draw[line width = 1pt, xshift = 6cm] (2,1) circle (0.065);

\newdimen\L
\L=0.35cm
       
\draw[  yshift = -0.5cm,xshift = 5.9cm,rotate = -18,] (0:\L)

\foreach \x in {72,144,...,360} {
      -- (\x:\L)
 } -- cycle (90:\L)  ;

\foreach \x in {0,72,144,...,360} {
     \draw[fill, yshift = -0.5cm,xshift = 5.9cm,rotate = -18 ] (\x:\L) circle (0.03);
 } ;

\draw[,  yshift = -0.5cm,xshift = 5.9cm,rotate = -18] (0:\L) -- (144:\L);
\draw[,  yshift = -0.5cm,xshift = 5.9cm,rotate = -18] (144:\L) -- (288:\L);

\draw[,  yshift = -0.5cm,xshift = 7.1cm,rotate = 126] (0:\L)
\foreach \x in {72,144,...,360} {
      -- (\x:\L)
 } -- cycle (90:\L)  ;

\foreach \x in {0,72,144,...,360} {
     \draw[fill, yshift = -0.5cm, xshift = 7.1cm,rotate = 126 ] (\x:\L) circle (0.03);
 } ;

\draw[,  yshift = -0.5cm,xshift = 7.1cm,rotate = 198] (0:\L) -- (144:\L);
\draw[,   yshift = -0.5cm,xshift = 7.1cm,rotate = 198] (144:\L) -- (288:\L);

\draw[,  yshift = 2.3cm,xshift = 5.6cm,rotate = 126] (0:\L)
\foreach \x in {72,144,...,360} {
      -- (\x:\L)
 } -- cycle (90:\L)  ;

\foreach \x in {0,72,144,...,360} {
     \draw[fill, yshift = 2.3cm, xshift = 5.6cm,rotate = 126] (\x:\L) circle (0.03);
 } ;
 
\draw[,  yshift = 2.3cm, xshift = 5.6cm,rotate = 126] (0:\L) -- (144:\L);
\draw[,  yshift = 2.3cm, xshift = 5.6cm,rotate = 126] (144:\L) -- (288:\L);

\draw[,  yshift = 1.0cm,xshift = 8.6cm,rotate = 54] (0:\L)
\foreach \x in {72,144,...,360} {
      -- (\x:\L)
 } -- cycle (90:\L)  ;

\foreach \x in {0,72,144,...,360} {
     \draw[fill, yshift = 1.0cm, xshift = 8.6cm,rotate = 54 ] (\x:\L) circle (0.03);
 } ;

\draw[, yshift = 1.0cm, xshift = 8.6cm,rotate = 54] (0:\L) -- (144:\L);
\draw[,  yshift = 1.0cm, xshift = 8.6cm,rotate = 54] (144:\L) -- (288:\L);

\draw[,  yshift = 2.1cm,xshift = 8.6cm,rotate = -90] (0:\L)
\foreach \x in {72,144,...,360} {
      -- (\x:\L)
 } -- cycle (90:\L)  ;

\foreach \x in {0,72,144,...,360} {
     \draw[fill, yshift = 2.1cm, xshift = 8.6cm,rotate = -90 ] (\x:\L) circle (0.03);
 } ;
 
\draw[,  yshift = 2.1cm,xshift = 8.6cm,rotate = -90] (0:\L) -- (144:\L);
\draw[,   yshift = 2.1cm,xshift = 8.6cm,rotate = -90] (144:\L) -- (288:\L);

\end{tikzpicture}

}
\end{centering}
\caption{The Rote--Santos--Streinu associahedron ${\rm RSS}_2 (\mathbf {g_0})$ 
with the coordinates of the vertices. This coincides with the realization of Buchstaber.}
\end{figure}

The facet of $\RSS (\mathbf g)$
defined by $y_j-y_i\geq g_{i,j}$
corresponds to the diagonal $\{i,j+1\}$
of an $(n+3)$-gon with vertices labeled 
in counterclockwise direction
from 0 to $n+2$. 
Rote, Santos and Streinu~\cite[Sec.~5.3]{RoSaSt03} notice that
$\RSS(\mathbf g)$ has exactly $n$ pairs of parallel facets, corresponding to the pairs of diagonals 
$(\{ 0,i+1 \},\{ i,n+2 \})$ for $1\leq i \leq n$, as illustrated in 
{\normalfont Figure~\ref{Parallel_ass^III}}.
 
\subsubsection{Affine equivalence}
Rote, Santos and Streinu stated in~\cite[Sec.~5.3]{RoSaSt03} 
that $\RSS (\mathbf g)$
is not affinely equivalent to neither the associahedron
as a secondary polytope nor the Chapoton--Fomin--Zelevinsky
associahedron of Section~\ref{section_clustercomplex}. 
%
Next we prove that $\RSS (\mathbf g)$
is normally isomorphic to $\Post (\mathbf a)$ and that this isomorphism 
induces an affine isomorphism between the Loday and Buchstaber specific
realizations.

\begin{theorem}
Let $\varphi$ be the affine transformation
\[
\begin{array}{cccc}
\varphi :  &    \R^{n+1}  &  \rightarrow &   \R ^n\\
  &  (x_1, \dots , x_{n+1}) & \rightarrow  & (y_1, \dots ,y_n)
\end{array}
\]
defined by $y_k=\sum_{i=1}^k (x_i-i)$.
Then
$\varphi$ maps 
$\Post (\mathbf a)$ bijectively to
$\RSS (\mathbf g)$, 
for $\mathbf g$ given by 
$g_{i,j} - \frac{(i+j+1)(j-i)}{2}= f_{i, j+1} (\mathbf a)$. 
In particular,  
$\varphi$
maps the Loday associahedron $\Post (\mathbf 1)$ 
to the Buchstaber associahedron $\RSS (\mathbf {g_0})$.
\end{theorem}

\begin{proof}
The result follows from the following computation
\[
\begin{array}{rcl}
 y_j-y_i & \geq  &    g_{i,j} \\
 (x_{i+1}+\cdots+x_j)+((i+1)+\cdots + j) & \geq  & g_{i,j}   \\
 x_{i+1}+\cdots+x_j & \geq  &  g_{i,j}-\frac{(i+j+1)(j-i)}2. \\[-11mm]{}
\end{array}
\]
\end{proof}

\subsection{The associahedron as a cluster polytope of type $A$}
\label{section_clustercomplex}

Cluster complexes are simplicial complexes associated to root systems and 
arose in the theory of cluster algebras
initiated by Fomin and Zelevinsky~\cite{FZ02,FoZe03}. 
In the initial papers by these two authors cluster complexes were realized only as complete fans, 
but these fans were shown to be polytopal in their subsequent work with Chapoton~\cite{CFZ02}.
The polytopes obtained are called \emph{generalized associahedra}
because the case of type $A_n$ yields to an associahedron.
We refer to~\cite{CFZ02},~\cite{FZ03} and~\cite{FR07} for more detailed presentations.   

\subsubsection{The cluster complex of type $A_n$}\label{sec:cluster_complex}
The \emph{root system of type} $A_n$ is the set 
$\Phi := \Phi (A_n) = \{ e_i-e_j, \ 1\leq i \neq j \leq n+1 \} \subset \R^{n+1}$. 
The \emph{simple roots} of type $A_n$ are the elements 
of the set $\Pi =\{ \alpha_i=e_i-e_{i+1}, i\in [n] \}$, 
the set of \emph{positive roots} is 
$\Phi_{>0}= \{e_i-e_j:i<j\}$, and the set of 
\emph{almost positive roots} is 
$\Phi_{\geq -1}:=\Phi_{>0}\cup -\Pi$.

In the theory of cluster algebras, a compatibility relation is introduced 
in the set of almost positive roots of a finite crystallographic root system
and the \emph{cluster complex} is defined as the simplicial complex of pairwise compatible roots~\cite{FoZe03,FZ03}.
For the root system of type $A_n$, 
there is a natural correspondence between 
the set $\Phi_{\geq-1}$ and the diagonals of the 
$(n+3)$-gon $P_{n+3}$ that sends compatible roots to non-crossing diagonals, and 
vice-versa~\cite[Prop.~3.14]{FZ03}. We take this property, which makes
the cluster complex  anti-isomorphic to the face complex of the associahedron, as a definition.

\begin{definition}[Cluster complex of type $A_n$]
Identify the negative simple roots $-\alpha_i$
with the diagonals on the \emph{snake} of $P_{n+3}$ 
illustrated in Figure~\ref{snake}. 
Each positive root is a consecutive sum
\[
\alpha_{ij}=\alpha_i+\alpha_{i+1}+\dots + \alpha_j, \hspace{1cm} 1\leq i \leq j \leq n,
\]
and thus can be identified with the unique diagonal of 
$P_{n+3}$ crossing the (consecutive) diagonals 
that correspond to $-\alpha_i, -\alpha_{i+1}, \dots, -\alpha_{j}$, and no others.
Two roots $\alpha$ and $\beta$ in 
$\Phi_{\geq -1}$ are called \textit{compatible} if 
their corresponding diagonals do not cross. 
The \emph{cluster complex} $\Delta (\Phi)$ of type $A_n$ 
is the clique complex of the compatibility relation on $\Phi_{\geq -1}$, 
i.e., the complex whose simplices correspond to 
the sets of almost positive roots that are pairwise compatible. 
Maximal simplices of $\Delta (\Phi)$ are called \emph{clusters}.
\end{definition}

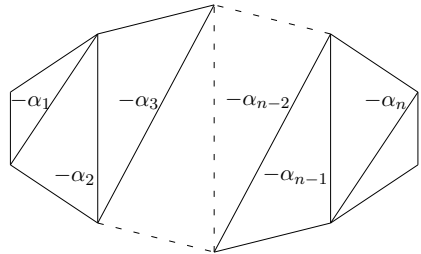
\begin{figure}[ht]
\begin{centering}
\resizebox{!}{33mm}{

\begin{tikzpicture}[ scale = 0.5];
\draw(-7,3) --(-4,1);
\draw[loosely dashed]  (-4,1)--(0,0);
\draw (0,0) -- (4,1);
\draw (4,1)--(7,3);
\draw(-7,3) -- (-7,5.5);
\draw(7,3) -- (7,5.5);
\draw(7,5.5) -- (4,7.5);
\draw(-7,5.5) -- (-4,7.5);
\draw(-4,7.5) -- (0,8.5);
\draw(4,7.5)[loosely dashed] -- (0,8.5);
\draw[loosely dashed](0,0)--(0,8.5);
\draw(0,0) -- (4,7.5);
\draw(-4,1) -- (0,8.5);
\draw(-4,1) -- (-4,7.5);
\draw(-7,3) -- (-4,7.5);
\draw(4,1) -- (4,7.5);
\draw(4,1) -- (7,5.5);

\draw (-6.3, 5.2) node {$-\alpha_1$};
\draw (-2.6, 5.2) node {$-\alpha_3$};
\draw (1.5, 5.2) node {$-\alpha_{n-2}$};
\draw (5.9, 5.2) node {$-\alpha_{n}$};

\draw (-4.8, 2.6) node {$-\alpha_{2}$};
\draw (2.8, 2.6) node {$-\alpha_{n-1}$};
\end{tikzpicture}

}
\end{centering}
\caption{Snake and negative roots of type $A_n$.}
\label{snake}
\end{figure}

\begin{theorem}[Fomin--Zelevinsky {\cite[Thms.~1.8, 1.10]{FZ03}}]\label{th_simpfan}
The simplicial cones $\R_{\geq 0} C$ generated by all clusters $C$ of type $A_n$ 
form a complete simplicial fan in the ambient space 
\[
\{(x_1,\dots ,x_{n+1})\in \R^{n+1}: x_1+\cdots +x_{n+1}=0 \}.
\]
\end{theorem}

\begin{theorem}[Chapoton--Fomin--Zelevinsky {\cite[Thm.~1.4]{CFZ02}}]
\label{theorem_cluster}
The simplicial fan in {\normalfont Theorem~\ref{th_simpfan}} is the normal fan of a simple $n$-dimensional polytope.
\end{theorem}

\begin{definition}
We call \emph{Chapoton--Fomin--Zelevinsky associahedron} $\CFZ(A_n)$ any polytope whose normal fan is the fan with maximal cones $\R_{\geq 0} C$ generated by all clusters $C$ of type $A_n$.
\end{definition}

A realization $\mathrm{CFZ}_2(A_2)$ is illustrated in
Figure~\ref{fig_assA2}; note how the facet normals correspond to the almost positive roots 
of~$A_2$. It is obvious from the definition of cluster complexes that
$\CFZ(A_n)$ has exactly $n$ pairs of parallel facets. These correspond to the pairs of roots $\{ \alpha_i,-\alpha_i \}$, for $i=1,\dots , n$, or, equivalently, to the pairs of diagonals $\{ \alpha_i,-\alpha_i \}$ as indicated in {\normalfont Figure~\ref{Parallel_ass^II}}. This is a particular case of Proposition~\ref{prop:triangmany_parallel}.

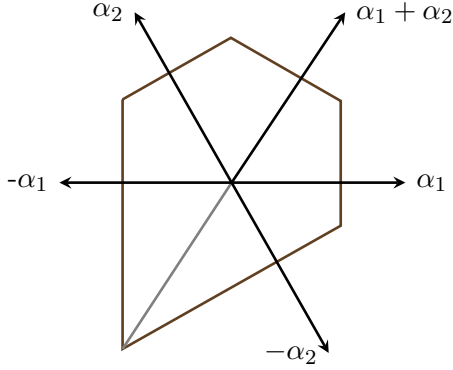
\begin{figure}[ht]
	\centering
	\begin{tikzpicture}[scale = 2,rotate =-90]
	\draw[line width = 1pt, color = brown!50!black] (2.08, 0.72) -- (3.74,0.72) -- (2.92,2.17) -- (2.09,2.17) -- (1.67,1.44) -- (2.08,0.72);

	\draw [line width = 1pt, color = gray] (3.74,0.72) -- (2.08+0.55333333,0.72+0.725);
	\draw [line width = 1pt, -stealth] (2.08+0.55333333,0.72+0.725)-- (2.08+0.55333333,2.6) node [anchor = west, scale = 1]{$\alpha_1$};
	\draw [line width = 1pt, -stealth] (2.08+0.55333333,0.72+0.725)-- (2.08+0.55333333,0.3)node [anchor = east, scale = 1]{-$\alpha_1$};
	\draw [line width = 1pt, -stealth] (2.08+0.55333333,0.72+0.725) -- (1.5,2.2)node [anchor = west, scale = 1]{$\alpha_1+\alpha_2$};

	\draw [line width = 1pt,-stealth] (2.08+0.55333333,0.72+0.725) -- (1.5,0.8)node [anchor = east, scale = 1]{$\alpha_2$};
	\draw [line width = 1pt,-stealth] (2.08+0.55333333,0.72+0.725) -- (3.76666666666, 2.09)node [anchor = east, scale = 1]{$-\alpha_2$};
	\end{tikzpicture}
	\caption{The complete simplicial fan of the cluster complex of type $A_2$ and an associahedron $\mathrm{CFZ}_2(A_2)$.}
	\label{fig_assA2}
\end{figure}

\begin{figure}[ht]
\begin{centering}
\resizebox{!}{33mm}{

\begin{tikzpicture}[scale = 0.5];
\draw[color = gray, xshift = -4cm ](-7,3) --(-4,1);
\draw[loosely dashed, color = gray, xshift = -4cm ]  (-4,1)--(0,0);
\draw [color = gray, xshift = -4cm ](0,0) -- (4,1);
\draw [color = gray, xshift = -4cm ](4,1)--(7,3);
\draw[color = gray, xshift = -4cm ](-7,3) -- (-7,5.5);
\draw[color = gray, xshift = -4cm ](7,3) -- (7,5.5);
\draw[color = gray, xshift = -4cm ](7,5.5) -- (4,7.5);
\draw[color = gray, xshift = -4cm ](-7,5.5) -- (-4,7.5);
\draw[color = gray, xshift = -4cm ](-4,7.5) -- (0,8.5);
\draw[loosely dashed,color = gray, xshift = -4cm ](4,7.5) -- (0,8.5);
\draw[loosely dashed, color = violet, xshift = -4cm ](0,0)--(0,8.5);
\draw[color = violet, xshift = -4cm ](0,0) -- (4,7.5);
\draw[color = violet, xshift = -4cm ](-4,1) -- (0,8.5);
\draw[color = violet, xshift = -4cm ](-4,1) -- (-4,7.5);
\draw[color = violet, xshift = -4cm ](-7,3) -- (-4,7.5);
\draw[color = violet, xshift = -4cm ](4,1) -- (4,7.5);
\draw[color = violet, xshift = -4cm ](4,1) -- (7,5.5);

\draw[xshift = -4cm] (-6.3, 5.2) node {$-\alpha_1$};
\draw[xshift = -4cm] (-2.5, 5.2) node {$-\alpha_3$};
\draw[xshift = -4cm] (1.5, 5.2) node {$-\alpha_{n-2}$};
\draw[xshift = -4cm] (6, 5.2) node {$-\alpha_{n}$};

\draw[xshift = -4cm] (-4.8, 2.6) node {$-\alpha_{2}$};
\draw[xshift = -4cm] (2.8, 2.6) node {$-\alpha_{n-1}$};

\draw[color = gray, xshift = 13cm](-7,3) --(-4,1);
\draw[loosely dashed, color = gray, xshift = 13cm]  (-4,1)--(0,0);
\draw [color = gray, xshift = 13cm](0,0) -- (4,1);
\draw [color = gray, xshift = 13cm](4,1)--(7,3);
\draw[color = gray, xshift = 13cm](-7,3) -- (-7,5.5);
\draw[color = gray, xshift = 13cm](7,3) -- (7,5.5);
\draw[color = gray, xshift = 13cm](7,5.5) -- (4,7.5);
\draw[color = gray, xshift = 13cm](-7,5.5) -- (-4,7.5);
\draw[color = gray, xshift = 13cm](-4,7.5) -- (0,8.5);
\draw[loosely dashed,color = gray, xshift = 13cm](4,7.5) -- (0,8.5);

\draw[loosely dashed, color = violet, xshift = 13cm](-4,1)--(4,7.5);
\draw[color = violet, xshift = 13cm](-4,1) -- (-7,5.5);
\draw[color = violet, xshift = 13cm](-7,3) -- (0,8.5);
\draw[color = violet, xshift = 13cm](0,0) -- (-4,7.5);
\draw[color = violet, xshift = 13cm](0,0) -- (7,5.5);
\draw[color = violet, xshift = 13cm](4,1) -- (0,8.5);
\draw[color = violet, xshift = 13cm](7,3) -- (4,7.5);

\draw[, xshift = 13cm] (-4.9, 5.4) node {$\alpha_2$};
\draw[, xshift = 13cm] (-1.4, 1.4) node {$\alpha_3$};
\draw[, xshift = 13cm] (6.0, 3.2) node {$\alpha_{n-1}$};
\draw[, xshift = 13cm] (4.6, 5.8) node {$\alpha_{n}$};

\draw[, xshift = 13cm] (-5.6, 2.7) node {$\alpha_{1}$};
\draw[, xshift = 13cm] (1.7,3.5) node {$\alpha_{n-2}$};
\end{tikzpicture}

}
\end{centering}
\caption{The diagonals of the $(n+3)$-gon that correspond to the pairs of parallel facets of $\CFZ(A_n)$.}
\label{Parallel_ass^II}
\end{figure}
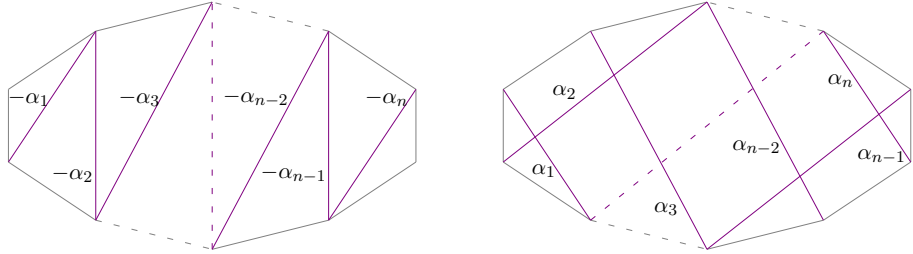

Theorem~\ref{th_simpfan} is the case of type $A_n$ of~\cite[Thm.~1.10]{FZ03}. 
Theorem~\ref{theorem_cluster} was conjectured by Fomin and Zelevinsky
\cite[Conj.~1.12]{FZ03} and subsequently proved by Chapoton, Fomin, and Zelevinsky~\cite{CFZ02}. 
For an explicit description by inequalities see~\cite[Cor.~1.9]{CFZ02}. 
These two theorems (for type $A$) are special cases of our Theorems~\ref{thm:triangmany_fan} and~\ref{thm:triangmany_regular}, proved in Section~\ref{sec:triang-many}.

\section{Exponentially many realizations, by Hohlweg--Lange} 
\label{sec:exponentially-many}

\subsection{The Hohlweg--Lange construction}

In this section we give a short description 
of the first, which we call ``type I'', exponential family of 
realizations of the associahedron,
obtained by Hohlweg and Lange in \cite{HoLa07}.
These associahedra are also generalized permutahedra and 
their construction uses ideas from
Shnider--Stemberg and Loday's constructions~\cite{ShniderSternberg93,Lo04}.
It was shown in \cite{BergeronHLT} that the number of normally non-isometric
realizations obtained this way is equal to the number of
sequences $\{+,-\}^{n-1}$ modulo reflection and reversal, which 
equals $2^{n-3}+2^{\lfloor \frac{n-3}{2} \rfloor}$
for $n\geq 3$ (see \cite[Sequence A005418]{Sloane}).
We show that classification by normal isomorphism yields the same number.
 
Let $\sigma \in \{+,-\}^{n-1}$ be a sequence of signs 
on the edges of an horizontal path on $n$ nodes.
We identify $n+3$ vertices $\{0,1,\dots , n+1,n+2\}$
with the signs of the sequence $\widetilde \sigma =\{+,-,\sigma,-,+\}$, and 
place them in convex position from left to right  
so that all positive vertices are above the horizontal path,
and all negative vertices are below it. 
These vertices form a convex $(n+3)$-gon 
that we call $P_{n+3}(\sigma)$. 
Figure \ref{fig_HLorientation}
illustrates the example $P_7(\{ +,-,+ \})$, where $n=4$.

\begin{figure}[ht]
\begin{centering}
\resizebox{!}{36mm}{
\begin{tikzpicture};

\filldraw (0,0) circle (0.15) node [anchor = north, scale = 1.5, yshift = -1mm]{$3$};
\filldraw (4,1.5) circle (0.15)node [anchor = north, scale = 1.5, yshift = -1mm]{$5$};
\filldraw (4.5,3.5) circle (0.15)node [anchor = south, scale = 1.5, yshift = 1mm]{$6$};
\filldraw (2,4.5) circle (0.15)node [anchor = south, scale = 1.5, yshift = 1mm]{$4$};
\filldraw (-4,1.5) circle (0.15)node [anchor = north, scale = 1.5, yshift = -1mm]{$1$};
\filldraw (-4.5,3.5) circle (0.15) node [anchor = south, scale = 1.5, yshift = 1mm]{$0$};
\filldraw (-2,4.5) circle (0.15)node [anchor = south, scale = 1.5, yshift = 1mm]{$2$};
\draw[line width = 1pt] (-2,4.5)--(-4.5,3.5)--(-4,1.5)-- (0,0) -- (4,1.5)--(4.5,3.5)--(2,4.5) -- (-2,4.5);

\draw[line width = 1pt] (-3,2.5)-- (-1,2.5) -- (1,2.5)--(3,2.5);
\filldraw (-3,2.5) circle (0.15);
\filldraw (3,2.5) circle (0.15);
\filldraw (1,2.5) circle (0.15);
\filldraw (-1,2.5) circle (0.15);

\draw (-2,2.1) node[scale = 1.5] {$+$};
\draw (2,2.1) node[scale = 1.5] {$+$};
\draw (0,2.1) node[scale = 1.5] {$-$};
\end{tikzpicture}

}
\end{centering}
\caption{$P_7(\{+,-,+\})$.}
\label{fig_HLorientation}
\end{figure}

The Hohlweg--Lange associahedra are obtained by removing certain facets of the standard $n$-dimensional permutahedron (\ref{eq:permutahedron}).
The facets that are removed depend on the choice of~$\sigma$, as follows.  

\begin{definition}\label{def:HLass}
	For a diagonal $ij$ ($i<j$) of $P_{n+3}(\sigma)$, we denote by $R_{ij}(\sigma)$ the 
	set of vertices strictly below it. We define the set $S_{ij}(\sigma)$ as the 
	result of 
	replacing $0$   by $i$ in $R_{ij}(\sigma)$ if $0\in R_{ij}(\sigma)$, and 
	replacing $n+2$ by $j$ if $n+2\in R_{ij}(\sigma)$.
The \emph{Hohlweg--Lange associahedron} $\AssI (\sigma)$ is the polytope
\[
\begin{array}{rl}
	\AssI (\sigma) = 
\big\{  (x_1,\dots x_{n+1}) \in \R^{n+1} : 
& \sum\limits_{i\in S_\delta (\sigma)} x_i \geq \tfrac12{|S_\delta(\sigma)|(|S_\delta(\sigma)| +1)}   \text{ for all diagonals } \delta,
\\ 
& x_1+\cdots+x_{n+1}=\tfrac{(n+1)(n+2)}2
\big\}.
\end{array}
\]
\end{definition} 

\begin{remark}\label{remark:change-signs}
If we interchange the first two signs and/or the last two signs
in $\widetilde \sigma =\{+,-,\sigma,-,+\}$   
the sets $S_\delta(\sigma)$ do not change and the 
construction produces the same associahedron $\AssI (\sigma)$.
\end{remark}

\begin{proposition}[Hohlweg--Lange {\cite[Prop.~1.3]{HoLa07}}] 
$\AssI(\sigma)$ is an $n$-dimensional associahedron. 
\end{proposition}

Hohlweg and Lange (\cite[Thm.~1.1]{HoLa07})  describe also the vertices of $\AssI(\sigma)$, extending Loday's rule (\cite[Thm.~1.1]{Lo04},~\cite[Cor.~8.2]{Po05}): 
To compute the $i$-th component of the vertex $(x_1,\dots, x_{n+1})$ corresponding to a triangulation $T$, look at the unique triangle $\tau_i$ of $T$ incident to vertex $i$ and whose interior intersects the vertical line through vertex $i$. The $n$ vertices in $T\setminus \tau$ fall into three  components: those to the left of $\tau$, those to the right, and those above $\tau$ (if $i$ is a negative vertex) or below $\tau$ (if $i$ is a positive vertex). Let $l_i$ and $r_i$ be the numbers of vertices to the left and right of $\tau$. Set:
\[
x_i =
 	\begin{cases}
 		(l_i+1)(r_i+1) & \text{if } \widetilde {\sigma}(i) = +  \\
 		n -  		(l_i+1)(r_i+1) & \text{if } \widetilde {\sigma}(i) = -.
	\end{cases}
\]
The reader can verify this rule for the Loday associahedron of Figure~\ref{fig:loday}.

\begin{proposition}[Hohlweg--Lange {\cite[Remarks~1.2,~4.3]{HoLa07}}]
\label{prop:special-cases-typeI}
$\AssI(\{-,-,\dots,-\})$ produces the Loday associahedron 
$\Post (\mathbf 1)$, and $\AssI(\{+,-,+,-,\dots\})$ is 
normally isomorphic to
the Chapoton--Fomin--Zelevinsky associahedron 
$\CFZ (A_n)$.
\end{proposition}

\begin{remark} 
\label{rem:minkowski-sum}
We defined the Loday associahedron as a Minkowski sum of certain faces of the standard simplex $\Delta_{[n+1]}$. The question arises whether such Minkowski sum descriptions exist for $\AssI (\sigma)$ in general. A partial answer is as follows:
The associahedra $\AssI (\sigma)$ are examples of generalized permutahedra (recall that a \emph{generalized permutahedron} is a polytope with facet normals contained in those of the standard permutahedron
such that the collection of right hand side parameters of the defining inequalities belongs to the deformation cone of the standard permutahedron; compare with the appendix by Postnikov et al. in~\cite{postnikov_faces_2008}). Generalized permutahedra include all the Minkowski sums $\sum_{S\subseteq [n+1]} a_S \Delta_S$ for which the coefficients $a_S$ are non-negative. Ardila et al.~\cite{ABD10} have shown that every generalized permutahedron admits a (unique) expression as a \emph{Minkowski sum and difference} of faces of the standard simplex.
These decompositions, for the case of  $\AssI (\sigma)$, are studied by Lange in~\cite{Lange11}.
A different decomposition arises from the work of Pilaud and Santos~\cite{PiSa11}, who show that the associahedra $\AssI (\sigma)$ are the ``brick polytopes" of
certain sorting networks. As such, they admit a decomposition as the Minkowski sum of the $\binom{n}{2}$ polytopes of the individual ``bricks". However, these summands need not be simplices.
\end{remark}

\begin{remark}\label{rem:gen-associahedra}
Hohlweg--Lange--Thomas \cite{HoLaTh11} provide a generalization of the Hohlweg--Lange construction to all finite Coxeter groups; for each Coxeter element $c$ (equivalently, for each orientation of the Coxeter graph) in a finite Coxeter system, they construct a realization of the corresponding generalized assiciahedron having as normal fan the $c$-Cambrian fan introduced earlier by Reading~\cite{Reading_cambrian} and discussed by Reading and Speyer  \cite{ReadingSpeyer09}. They call this polytope the $c$-generalized associahedron. For types $A$ and $B$, this specializes to the Hohlweg--Lange associahedra and cyclohedra. 

 A common generalization of $c$-generalized associahedra and brick polytopes (see previous remark) 
 is introduced by Pilaud--Stump~\cite{PilaudStump}. Another interesting construction of the Hohlweg--Lange--Thomas $c$-generalized associahedra is obtained by Stella in~\cite{Stella}.
\end{remark}

\subsection{Normal facet vectors, and normal isomorphism}
\label{subsec:AssI_parallel}
\label{subsec:AssI_0-1}
The Hohlweg--Lange associahedra satisfy properties (1) and (2) mentioned in the introduction: they have $n$ pairs of parallel facets and in the basis given by the normals to those facets all normal facet vectors are in $\{-1,0,1\}^n$. To see this,
we denote $e_S$ the characteristic vector of each subset $S\subset[n+1]$. By definition, the normal vectors of $\AssI(\sigma)$ are the characteristic vectors of the sets $S_{ij}(\sigma)$ associated to the different diagonals of $P_{n+3}(\sigma)$. 
But, as mentioned in the proof of Proposition \ref{prop:special-cases-typeI}, these vectors are to be considered modulo $e_{[n+1]}=(1,\dots,1)$. In particular, we have $e_S+ e_{\overline S}=0$ if $\overline S = [n+1]\setminus S$.

Although the following result is implicit in~\cite{HoLa07}, we include a proof because our description of the diagonals corresponding to parallel facets is a bit more explicit and will be used in the proofs of Theorems~\ref{thm:classification_HL} and~\ref{theo:almost-disjoint}.

\begin{proposition}[Hohlweg--Lange {\cite[Lem.~2.2 and Cor.~2.3]{HoLa07}}]
\label{prop:parallel_typeI}
$\AssI (\sigma)$ has exactly $n$ pairs of parallel facets, whose normal vectors are $e_{[j]}$ and $e_{\overline{[j]}}$ for $j=1,\dots,n$.
They correspond to the diagonals of the quadrilaterals 
with vertices $\{i,j,j+1,k\}$ for $j=1,\dots ,n$, where
\[
i= \max \{ 0\leq r < j : \rm{ sign}(r) \cdot \rm{sign}(j)=-  \} 
\]
\[
k=\min \{ j+1<r \leq n+2 : \rm{sign} (r) \cdot \rm{sign}(j+1) = - \}
\]
\end{proposition}

Hohlweg and Lange~\cite[Cor.~2.3]{HoLa07} also remark that the facets with normals $e_{[j]}$, $j\in [n+1]$ intersect at a vertex of 
$\AssI (\sigma)$ and the ones with normals $e_{\overline {[j]}}$, $j\in [n+1]$ intersect at another (opposite) one. That is to say, the $2n$ corresponding diagonals of $P_{n+3}(\sigma)$ form two triangulations, as was the case in Figure~\ref{Parallel_ass^III}.

\begin{proof}
Two diagonals $\delta$ and $\delta'$ correspond to
two parallel facets of $\AssI(\sigma)$ if and only if
the sets $S_\delta$ and $S_{\delta'}$ are complementary. By the definition of $S_{\delta}$, the only way this can happen is when $\delta$ and $\delta'$ are two crossing diagonals of opposite slope signs and such that the quadrilateral containing them uses an edge from the lower chain of $P_{n+3}(\sigma)$ and an edge of the upper chain. This is the case described in the statement, and it is easy to check that the corresponding $S_{\delta}$ and $S_{\delta'}$ are, respectively, $\{1,\dots,j\}$ and $\{j+1,\dots,n+1\}$.
\end{proof}

That all other normals have coordinates in $\{-1,0,+1\}^n$ when expressed in the basis $\{e_{[j]}, j\in [n]\}$ follows trivially from the following equation, valid for every $S\subset [n+1]$:
\[
e_S = \sum_{j\in S\atop j+1\not\in S} e_{[j]} - \sum_{j+1\in S\atop j\not\in S} e_{[j]}.
\]
\begin{corollary}
\label{coro:AssI_0-1}
With respect to the basis $\{e_{[1]},\dots, e_{[n]}\}$, (and considered modulo $e_{[n+1]}$), the normal vectors of $\AssI(\sigma)$ are all in $\{0,+1,-1\}^n$ and include $\{\pm e_{[1]},\dots, \pm e_{[n]}\}$.
\end{corollary}

As was pointed out to us by one of the referees, this result can be stated in the language of Coxeter combinatorics as follows: the normal vectors of the Hohlweg--Lange associahedra are weights of the root system of type~$A$; the parallel facets correspond to fundamental weights and every weight of type~$A$ can be written as a linear combination, with coefficients in $\{0,-1,1\}$, of the fundamental weights.

We now use parallel facets to classify Hohlweg--Lange associahedra.
For a sequence $\sigma\in\{-,+\}^{n-1}$ we define the reflection of $\sigma$ as the sequence $-\sigma$,
and the reversal $\sigma^t$ as the result of reversing the order of coordinates in $\sigma$.

\begin{theorem}
\label{thm:classification_HL}
Let $\sigma_1,\sigma_2 \in \{+,-\}^{n-1}$. Then $\AssI (\sigma_1)$ and $\AssI (\sigma_2)$ 
are normally isomorphic if and only if 
$\sigma_2$ can be obtained from $\sigma_1$ by
reflections and reversals.   
\end{theorem}

\begin{proof}
Suppose there is a linear isomorphism between the normal fans of 
$\AssI (\sigma_1)$ and $\AssI (\sigma_2)$. 
It induces an automorphism of the face lattice of the associahedron
that, by Lemma~\ref{lemma:automorphism_rotation-reflection}, 
corresponds to a certain reflection-rotation of the polygon. We denote this reflection-rotation by  
$\varphi : P_{n+3}(\sigma_1) \rightarrow P_{n+3}(\sigma_2)$.
Any linear isomorphism of the normal fans preserves
the property of a pair of facets being parallel, so  
$\varphi$ maps the ``parallel" pairs of diagonals of $P_{n+3}(\sigma_1)$,
to the ``parallel" pairs of diagonals of $P_{n+3}(\sigma_2)$.  
Furthermore, for both realizations there are exactly four diagonals 
that cross at least one diagonal of every parallel pair; they are 
$\{0,n+1\},\{0,n+2\},\{1,n+1\}$ and $\{1,n+2\}$. The set of these
four diagonals is also preserved under $\varphi$. 
This is possible only if $\varphi$ is a 
reflection-rotation of $P_{n+3}(\sigma_1)$, which
corresponds to a reflection-reversal of the sequence 
$\widetilde {\sigma_1}=\{+,-,\sigma_1,-,+\}$. 

It remains to be proved that $\AssI (\sigma)$ is normally-isomorphic 
to both $\AssI (-\sigma )$ and $\AssI (\sigma^t)$. 
The isomorphism between the normal fans of $\AssI(\sigma)$ and $\AssI(-\sigma)$
is given by multiplication by $-1$, since 
$S_\delta(-\sigma) = [n]- S_\delta(\sigma)$. The isomorphism 
between the normal fans of $\AssI(\sigma)$ and $\AssI(\sigma^t)$
is given by the permutation of coordinates $\tau(i)=n+1-i$, as 
$S_\delta(\sigma^t) = \tau(S_\delta(\sigma)) $. 
\end{proof}

In particular, putting together Proposition~\ref{prop:special-cases-typeI} and Theorem~\ref{thm:classification_HL}, one obtains
that the Loday associahedron is not normally isomorphic to the Chapoton--Fomin--Zelevinsky associahedron, for $n\ge 3$.

\begin{remark}
\label{rem:gen-associahedra-class}
Bergeron, Hohlweg, Lange and Thomas~\cite[Thm.~2.3]{BergeronHLT}  classify the Hohlweg--Lange--Thomas $c$-generalized associahedra up to isometry, and also up to isometry of normal fans \cite[Cor.~2.6]{BergeronHLT}. 
Even if those classifications yield the same result as ours, they do not automatically imply it.
As an example of why these classifications are potentially different, consider the rhombus obtained by removing two opposite facets of a regular hexagon. This generalized permutahedron admits a normal automorphism that is not an isometry of normal fans: the affine map that interchanges acute and obtuse angles. That is to say, the fact that Theorem~\ref{thm:classification_HL}
yields the same classes as the classification in~\cite{BergeronHLT} implies the following result for Hohlweg--Lange associahedra, which fails for other generalized permutahedra:

\begin{proposition}
\label{prop:isometry}
$\AssI(\sigma)$ and $\AssI(\sigma')$  are normally isomorphic if and only if they are isometric.
\end{proposition}

We do not know whether the same is true for $c$-generalized associahedra in other types.
\end{remark}

\section{Catalan many realizations, by Santos}
\label{sec:triang-many}

In this section we describe a generalization of the Chapoton--Fomin--Zelevinsky  construction of the associahedron (Section~\ref{section_clustercomplex}). It was
originally presented at a conference in 2004 \cite{Sa04}, but unpublished until now.
We prove that the number of normally non-isomorphic realizations obtained
this way, our ``type II exponential family'', is equal to the number of triangulations of an $(n+3)$-gon modulo reflections and rotations. This number equals
\[
\tfrac{1}{2(n+3) }C_{n+1}+ \tfrac{1}{4}C_{(n+1)/2} + \tfrac{1}{2}C_{\lfloor (n+1)/2 \rfloor} + \tfrac{1}{3}C_{n/3},
\]
where $C_n=\frac{1}{n+1} \binom{2n}{n}$ for $n\in\mathbb Z$ and $C_n=0$ otherwise.
Interest in this sequence goes back to Motzkin (1948) \cite{Mo48},~\cite[Sequence A000207]{Sloane}. 

Let $\alpha_1,\dots,\alpha_n$ denote a linear basis of an $n$-dimensional real vector space $V\cong \R^n$, and let $T_0$ be a certain triangulation of the $(n+3)$-gon, fixed once and for all throughout the construction. We call $T_0$ the \emph{seed triangulation}. The CFZ associahedron will arise as the special case where $V=\{(x_1,\dots,x_{n+1})\in \R^{n+1}:\sum x_i=0\}$, $\alpha_i=e_i-e_{i+1}$, and $T_0$ is the \emph{snake triangulation} of Figure~\ref{snake}.

Let $\{\delta_1,\dots,\delta_n\}$ denote the $n$ diagonals present in the seed triangulation $T_0$. To each diagonal $pq$ out of the $\frac{n(n+3)}{2}$ possible diagonals of the $(n+3)$-gon we associate a vector $v_{pq}$ as follows:
\begin{compactitem}[$\circ$]
\item If $pq=\delta_i$ for some $i$ (that is, if $pq$ is used in $T_0$) then let $v_{pq}= - \alpha_i$.
\item If $pq\not\in T_0$ then let
\[
v_{pq}:=\sum_{pq \text{ crosses } \delta_i} \alpha_i.
\]
\end{compactitem}

As a running example, consider the triangulation $\{123,345,156,135\}$ of a hexagon with its vertices labelled cyclically. Let $\delta_1=13$, $\delta_2=35$ and $\delta_3=15$. Written with respect to the basis $\{\alpha_1,\alpha_2,\alpha_3\}$ the nine vectors $v_{pq}$ that we get are as follows (see Figure~\ref{fig:hexagon}):

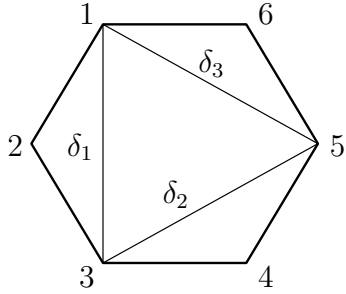
\begin{figure}[ht]
	\centering

\begin{picture}(0,0)%
\includegraphics{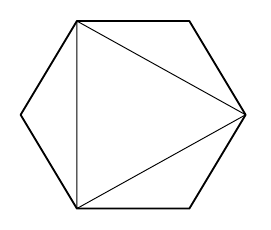}%
\end{picture}%
\setlength{\unitlength}{3947sp}%
\begingroup\makeatletter\ifx\SetFigFont\undefined%
\gdef\SetFigFont#1#2#3#4#5{%
  \reset@font\fontsize{#1}{#2pt}%
  \fontfamily{#3}\fontseries{#4}\fontshape{#5}%
  \selectfont}%
\fi\endgroup%
\begin{picture}(2163,1830)(4186,-4276)
\put(4201,-3436){\makebox(0,0)[lb]{\smash{{\SetFigFont{12}{14.4}{\rmdefault}{\mddefault}{\updefault}{\color[rgb]{0,0,0}2}%
}}}}
\put(4651,-4261){\makebox(0,0)[lb]{\smash{{\SetFigFont{12}{14.4}{\rmdefault}{\mddefault}{\updefault}{\color[rgb]{0,0,0}3}%
}}}}
\put(6226,-3436){\makebox(0,0)[lb]{\smash{{\SetFigFont{12}{14.4}{\rmdefault}{\mddefault}{\updefault}{\color[rgb]{0,0,0}5}%
}}}}
\put(5776,-2611){\makebox(0,0)[lb]{\smash{{\SetFigFont{12}{14.4}{\rmdefault}{\mddefault}{\updefault}{\color[rgb]{0,0,0}6}%
}}}}
\put(5776,-4261){\makebox(0,0)[lb]{\smash{{\SetFigFont{12}{14.4}{\rmdefault}{\mddefault}{\updefault}{\color[rgb]{0,0,0}4}%
}}}}
\put(4576,-3436){\makebox(0,0)[lb]{\smash{{\SetFigFont{12}{14.4}{\rmdefault}{\mddefault}{\updefault}{\color[rgb]{0,0,0}$\delta_1$}%
}}}}
\put(5401,-2911){\makebox(0,0)[lb]{\smash{{\SetFigFont{12}{14.4}{\rmdefault}{\mddefault}{\updefault}{\color[rgb]{0,0,0}$\delta_3$}%
}}}}
\put(5176,-3736){\makebox(0,0)[lb]{\smash{{\SetFigFont{12}{14.4}{\rmdefault}{\mddefault}{\updefault}{\color[rgb]{0,0,0}$\delta_2$}%
}}}}
\put(4651,-2611){\makebox(0,0)[lb]{\smash{{\SetFigFont{12}{14.4}{\rmdefault}{\mddefault}{\updefault}{\color[rgb]{0,0,0}1}%
}}}}
\end{picture}%
	\caption{A seed triangulation for Santos' construction.}
	\label{fig:hexagon}
\end{figure}
\[
\begin{array}{lll}
v_{13}=-\alpha_1=(-1,0,0), & v_{35}=-\alpha_2 =(0,-1,0), & v_{15}=-\alpha_3=(0,0,-1), \cr
v_{25}=\alpha_1=(1,0,0), & v_{14}=\alpha_2 =(0,1,0), & v_{36}=\alpha_3=(0,0,1), \cr
v_{46}=\alpha_2+\alpha_3=(0,1,1), & v_{26}=\alpha_1+\alpha_3 =(1,0,1), & v_{24}=\alpha_1+\alpha_2=(1,1,0). \cr
\end{array}
\]

With a slight abuse of notation we denote with the same symbol a subset of diagonals of the polygon and the set of vectors associated with them. For example, $\R_{\ge0} T_0=\R_{\ge0} \{-\alpha_1,\dots,-\alpha_n\}$ is the negative orthant in $V$ (with respect to the basis $[\alpha_i]_i$). 
More generally, for each triangulation $T$ of the $(n+3)$-gon consider the cone $\R_{\ge0} T$. We claim the following generalizations of Theorems~\ref{th_simpfan} and~\ref{theorem_cluster}:

\begin{theorem}\label{thm:triangmany_fan}
The simplicial cones $\R_{\geq 0} T$ generated by all triangulations $T$ of the $(n+3)$-gon
form a complete simplicial fan ${\mathcal F}_{T_0}$ in the ambient space $V$.
\end{theorem}

\begin{theorem}\label{thm:triangmany_regular}
This fan ${\mathcal F}_{T_0}$ is the normal fan of an $n$-dimensional associahedron.
\end{theorem}

Our proofs are based on the understanding of a complete simplicial fan as a \emph{triangulation of a totally cyclic vector configuration}, which makes \emph{regular triangulations} correspond to normal fans of simple polytopes (see~\cite[Sects.~2.5, 9.5]{LoRaSa10}, and compare our two statements to steps (1) and (2) in~\cite[p.~503]{LoRaSa10}). Incidentally, this method is  illustrated there by constructing the normal fan of the Loday associahedron and showing its polytopality.

\subsection{Proof of Theorem~\ref{thm:triangmany_fan}}
The statement follows from the following two claims:
\begin{compactenum}[(1)]
\item $\R_{\geq 0} T_0$ is a simplicial cone and is the only cone in ${\mathcal F}_{T_0}$ that intersects (the interior of) the negative orthant.
\item If $T_1$ and $T_2$ are two triangulations that differ by a flip, let  $v_1\in T_1$ and $v_2\in T_2$ be the diagonals removed and inserted by the flip. That is,
$T_1\setminus T_2 =\{v_1\}$ and $T_2\setminus T_1= \{v_2\}$. Then there is a linear dependence in $T_1\cup T_2$ which has coefficients of the same sign (and different from zero) in the elements $v_1$ and $v_2$.
\end{compactenum}
The first assertion is obvious, and the second one is Lemma~\ref{lemma:triangmany_fan} below. Before proving it let us argue why these two assertions imply Theorem~\ref{thm:triangmany_fan}. Suppose that we have two triangulations $T_1$ and $T_2$ related by a flip as in the second assertion, and suppose that we already know that one of them, say $T_1$, spans a full-dimensional cone (that is, we know that $T_1$ considered as a set of vectors is independent). Then assertion (2) implies that $T_2$ spans a full-dimensonal cone as well and that $\R_{\geq 0}T_1$ and $\R_{\geq 0}T_2$ lie in opposite sides of their common facet $\R_{\geq 0}(T_1\cap T_2)$. This, together with the fact that there is some part of $V$ covered by exactly one cone (which is why we need assertion (1)) implies that we have a complete fan. (See, for example, \cite[Cor.~4.5.20]{LoRaSa10}, where assertion (2) is a special case of ``property (ICoP)'' and assertion (1) a special case of ``property (IPP)''.)

\begin{lemma}\label{lemma:triangmany_fan}
Let $T_1$ and $T_2$ be two triangulations that differ by a flip, and let  $v_1$ and $v_2$ be the diagonals removed and inserted by the flip from $T_1$ to $T_2$, respectively (that is,
$T_1\setminus T_2 =\{v_1\}$ and $T_2\setminus T_1= \{v_2\}$). Then there is a linear dependence in $T_1\cup T_2$ which has coefficients of the same sign in the elements $v_1$ and $v_2$.
\end{lemma}

\begin{proof}
Let $p$, $q$, $r$ and $s$ be the four points involved by the two diagonals $v_1$ and $v_2$, in cyclic order. That is, the diagonals removed and inserted are $pr$ and $qs$. We claim that one (and exactly one) of the following things occurs (see Figure~\ref{fig:triangmany_cases}):
\begin{compactenum}[(a)]
\item There is a diagonal in the seed triangulation $T_0$ that crosses two opposite edges of the quadrilateral $pqrs$.
\item One of $pr$ and $qs$ is used in the seed triangulation $T_0$.
\item There is a triangle $abc$ in $T_0$ with a vertex in $pqrs$  and the opposite edge crossing two sides of $pqrs$ (that is, without loss of generality $p=a$ and $bc$ crosses both $qr$ and $rs$).
\item There is a triangle $abc$ in $T_0$ with an edge in common with $pqrs$ and with the other two edges of the triangle crossing the opposite edge of the quadrilateral (that is, without loss of generality, $p = a$, $q=b$  and $rs$ crosses both $ac$  and $bc$).
\end{compactenum}

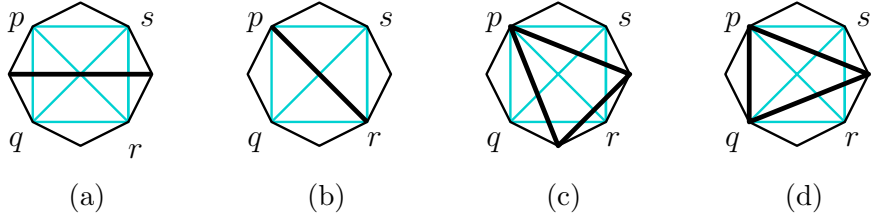
\begin{figure}[ht]
	\centering
	
\begin{picture}(0,0)%
\includegraphics{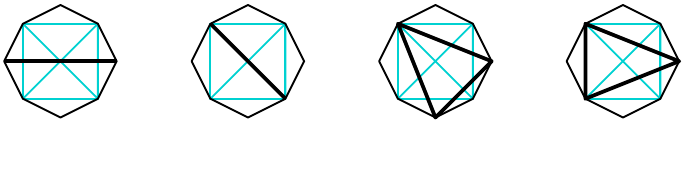}%
\end{picture}%
\setlength{\unitlength}{3947sp}%
\begingroup\makeatletter\ifx\SetFigFont\undefined%
\gdef\SetFigFont#1#2#3#4#5{%
  \reset@font\fontsize{#1}{#2pt}%
  \fontfamily{#3}\fontseries{#4}\fontshape{#5}%
  \selectfont}%
\fi\endgroup%
\begin{picture}(5466,1377)(718,-3955)
\put(751,-2761){\makebox(0,0)[lb]{\smash{{\SetFigFont{12}{14.4}{\rmdefault}{\mddefault}{\updefault}{\color[rgb]{0,0,0}$p$}%
}}}}
\put(1576,-2761){\makebox(0,0)[lb]{\smash{{\SetFigFont{12}{14.4}{\rmdefault}{\mddefault}{\updefault}{\color[rgb]{0,0,0}$s$}%
}}}}
\put(2251,-2761){\makebox(0,0)[lb]{\smash{{\SetFigFont{12}{14.4}{\rmdefault}{\mddefault}{\updefault}{\color[rgb]{0,0,0}$p$}%
}}}}
\put(3076,-2761){\makebox(0,0)[lb]{\smash{{\SetFigFont{12}{14.4}{\rmdefault}{\mddefault}{\updefault}{\color[rgb]{0,0,0}$s$}%
}}}}
\put(3751,-2761){\makebox(0,0)[lb]{\smash{{\SetFigFont{12}{14.4}{\rmdefault}{\mddefault}{\updefault}{\color[rgb]{0,0,0}$p$}%
}}}}
\put(4576,-2761){\makebox(0,0)[lb]{\smash{{\SetFigFont{12}{14.4}{\rmdefault}{\mddefault}{\updefault}{\color[rgb]{0,0,0}$s$}%
}}}}
\put(5251,-2761){\makebox(0,0)[lb]{\smash{{\SetFigFont{12}{14.4}{\rmdefault}{\mddefault}{\updefault}{\color[rgb]{0,0,0}$p$}%
}}}}
\put(6076,-2761){\makebox(0,0)[lb]{\smash{{\SetFigFont{12}{14.4}{\rmdefault}{\mddefault}{\updefault}{\color[rgb]{0,0,0}$s$}%
}}}}
\put(751,-3511){\makebox(0,0)[lb]{\smash{{\SetFigFont{12}{14.4}{\rmdefault}{\mddefault}{\updefault}{\color[rgb]{0,0,0}$q$}%
}}}}
\put(1501,-3586){\makebox(0,0)[lb]{\smash{{\SetFigFont{12}{14.4}{\rmdefault}{\mddefault}{\updefault}{\color[rgb]{0,0,0}$r$}%
}}}}
\put(2251,-3511){\makebox(0,0)[lb]{\smash{{\SetFigFont{12}{14.4}{\rmdefault}{\mddefault}{\updefault}{\color[rgb]{0,0,0}$q$}%
}}}}
\put(3001,-3511){\makebox(0,0)[lb]{\smash{{\SetFigFont{12}{14.4}{\rmdefault}{\mddefault}{\updefault}{\color[rgb]{0,0,0}$r$}%
}}}}
\put(3751,-3511){\makebox(0,0)[lb]{\smash{{\SetFigFont{12}{14.4}{\rmdefault}{\mddefault}{\updefault}{\color[rgb]{0,0,0}$q$}%
}}}}
\put(4501,-3511){\makebox(0,0)[lb]{\smash{{\SetFigFont{12}{14.4}{\rmdefault}{\mddefault}{\updefault}{\color[rgb]{0,0,0}$r$}%
}}}}
\put(5251,-3511){\makebox(0,0)[lb]{\smash{{\SetFigFont{12}{14.4}{\rmdefault}{\mddefault}{\updefault}{\color[rgb]{0,0,0}$q$}%
}}}}
\put(6001,-3511){\makebox(0,0)[lb]{\smash{{\SetFigFont{12}{14.4}{\rmdefault}{\mddefault}{\updefault}{\color[rgb]{0,0,0}$r$}%
}}}}
\put(1126,-3886){\makebox(0,0)[lb]{\smash{{\SetFigFont{11}{13.2}{\rmdefault}{\mddefault}{\updefault}{\color[rgb]{0,0,0}(a)}%
}}}}
\put(2626,-3886){\makebox(0,0)[lb]{\smash{{\SetFigFont{11}{13.2}{\rmdefault}{\mddefault}{\updefault}{\color[rgb]{0,0,0}(b)}%
}}}}
\put(4126,-3886){\makebox(0,0)[lb]{\smash{{\SetFigFont{11}{13.2}{\rmdefault}{\mddefault}{\updefault}{\color[rgb]{0,0,0}(c)}%
}}}}
\put(5626,-3886){\makebox(0,0)[lb]{\smash{{\SetFigFont{11}{13.2}{\rmdefault}{\mddefault}{\updefault}{\color[rgb]{0,0,0}(d)}%
}}}}
\end{picture}%
	\caption{The four cases in the proof of Lemma~\ref{lemma:triangmany_fan}.}
	\label{fig:triangmany_cases}
\end{figure}

To prove that one of the four things occurs we argue as follows. It is well-known that in any triangulation of a $k$-gon one can ``contract a boundary edge'' to get a triangulation of a $(k-1)$-gon. Doing that in all the boundary edges of the seed triangulation $T_0$ except those incident to either $p$, $q$, $r$ or $s$ we get a triangulation $\widetilde{T_0}$ of a polygon $\widetilde P$ with at most eight vertices: the four vertices $p$, $q$, $r$ and $s$ and at most one extra vertex between each two of them. We embed $\widetilde P$ having 
as vertex a subset of the vertices of a regular octagon, with $pqrs$ forming a square.
We now look at the position of the center of the octagon $\widetilde P$ with respect to the triangulation  $\widetilde{T_0}$: If it lies in the interior of an edge, then this edge is a diameter of the octagon and we are in cases (a) or (b).
If it lies in the interior of a triangle of $\widetilde{T_0}$, then we are in cases (c) or (d). See Figure~\ref{fig:triangmany_cases} again.

Now we show explicitly the linear dependences involved in $T_1\cup T_2$ in each case. 
\begin{compactenum}[(a)]
\item  Suppose $T_0$ has a diagonal crossing $pq$ and $rs$. Then
\begin{equation}\label{eq:triangmany_a}
v_{pr} + v_{qs} = v_{pq} + v_{rs},
\end{equation}
because every diagonal of $T_0$ intersecting the two (respectively, one; respectively none) 
of $pr$ and $qs$ intersects also the two  (respectively, one; respectively none) of $pq$ and $rs$.

\item If $T_0$ contains the diagonal $pr$, let $a$ and $b$ be vertices joined to $pr$ in $T_0$, with $a$ on the side of $q$ and $b$ on the side of $s$. We define the following vectors $w_a$ and $w_b$:
\begin{compactitem}[$\circ$]
\item $w_a$ equals $0$, $v_{pq}$ or $v_{qr}$ depending on whether $a$ equals $q$, lies between $p$ and $q$, or lies between $q$ and $r$.
\item $w_b$ equals $0$, $v_{ps}$ or $v_{rs}$ depending on whether $a$ equals $s$, lies between $p$ and $s$, or lies between $s$ and $r$.
\end{compactitem}
We claim that in the nine cases we have the equality
\begin{equation}\label{eq:triangmany_b}
v_{pr} + v_{qs} = w_{a} + w_{b}.
\end{equation}
This is so because $v_{pr} + v_{qs}$ now equals the sum of the $\alpha_i$'s corresponding to the diagonals of $T_0\setminus\{pr\}$ crossing $qs$, and we have that:
\begin{compactitem}[$\circ$]
\item The diagonals of $T_0$ crossing $qs$ in the $q$-side of $pr$ are none, the same as those crossing $pq$, or the same as those crossing $qr$ in the three cases of the definition of $w_a$, and
\item The diagonals of $T_0$ crossing $qs$ in the $s$-side of $pr$ are none, the same as those crossing $ps$, or the same as those crossing $rs$ in the three cases of the definition of $w_b$
\end{compactitem}

\item  If $T_0$ contains a triangle $pbc$ with $bc$ crossing both $qr$ and $rs$ then we have the equality
\begin{equation}\label{eq:triangmany_c}
2v_{pr} + v_{qs} = v_{qr} + v_{rs},
\end{equation}
because in this case the diagonals of $T_0$ crossing $pr$ are those crossing both $qr$ and $rs$, while the ones crossing $qs$ are those crossing one, but not both, of  $qr$ and $rs$.

\item If $T_0$ contains a triangle $pqc$ with $rs$ crossing both $pc$ and $qc$ then we have the equality
\begin{equation}\label{eq:triangmany_d}
v_{pr}+ v_{qs}= v_{rs}
\end{equation}
because the diagonals of $T_0$ crossing $rs$ are the disjoint union of those crossing $pr$ and those crossing $qs$.
\end{compactenum}
\vskip-5mm
\end{proof}

Observe that when $T_0$ is a snake triangulation (the CFZ case) or, more generally, when the dual tree of $T_0$ is a path, cases (c) and (d) do not occur.

\subsection{Proof of Theorem~\ref{thm:triangmany_regular}}

Once we know ${\mathcal F}_{T_0}$ is a complete simplicial fan, its being the normal fan of a simple polytope can be expressed as the feasibility of a system of linear inequalities. This can be done in several ways (compare, e.g., Theorem~3.1 in~\cite{HoLaTh11}). We choose the following one,
related to the understanding of complete simplicial fans as \emph{triangulations of vector configurations}~(see~\cite[Sec.~9.5]{LoRaSa10}).

\begin{lemma}
\label{lemma:polytopal}
Let ${\mathcal F}$ be a complete simplicial fan in a real vector space $V$ and let $A$ be the set of generators of ${\mathcal F}$ (more precisely, $A$ has one generator of each ray of ${\mathcal F}$). Then the following conditions are equivalent:
\begin{compactenum}[\rm(1)]
\item ${\mathcal F}$ is the normal fan of a polytope.

\item There is a map $\omega: A\to \R_{>0}$ such that for every pair $(C_1, C_2)$ of maximal adjacent cones of ${\mathcal F}$ the following happens: Let $\lambda:A\to \R$ be the (unique, up to a scalar multiple) linear dependence with support in $C_1\cup C_2$, with its sign chosen so that $\lambda$ is positive in the generators of $C_1\setminus C_2$ and $C_2\setminus C_1$. Then the scalar product $\lambda\cdot \omega = \sum_v \lambda(v) \omega(v)$ is strictly positive.
\end{compactenum}
\end{lemma}

\begin{proof}
One short proof of the lemma is that both conditions are equivalent to ``${\mathcal F}$ is a regular triangulation of the vector configuration $A$''~(See, e.~g.,~Corollary 9.5.3~\cite{LoRaSa10}). But let us show a more explicit proof of the implication from (2) to (1), which is the one we need. What we are going to show is that if such an $\omega$ exists and if we consider the 
set of points
\[
\widetilde A:=\big\{\tfrac{v}{\omega(v)}: v\in A\big\},
\]
then the convex hull of $\widetilde A$ is a simplicial polytope having  ${\mathcal F}$ as its \emph{central fan}. (We think of $\widetilde A$ as points in an affine space, rather than as vectors in a vector space.) Hence  ${\mathcal F}$ is the normal fan of the polar of $\conv(\widetilde A)$ (see, e.~g., \cite[Sec.~7.1]{Zi}).

To show the claim on $\conv(\widetilde A)$ we argue as follows. Consider the simplicial complex $\Delta$ with vertex set $\widetilde A$ obtained by embedding the face lattice of ${\mathcal F}$ in it. That is, for each cone~$C$ of ${\mathcal F}$ we consider the simplex with vertex set in $\widetilde A$ corresponding to the generators of~$C$. Since ${\mathcal F}$ is a complete fan and since the elements of $\widetilde A$ are generators for its rays (they are positive scalings of the elements of $A$), $\Delta$ is the boundary of a star-shaped polyhedron with the origin in its kernel. The only thing left to be shown is that this polyhedron is strictly convex, that is, that for any two adjacent maximal simplices $\sigma_1$ and $\sigma_2$ the origin lies in the same side of $\sigma_1$ as $\sigma_2\setminus \sigma_1$. Equivalently, if we understand (the vertices of) $\sigma_1$ and $\sigma_2$ as subsets of $\widetilde A$,
we have to show that the unique affine dependence between the points $\{O\}\cup\sigma_1\cup\sigma_2$ has opposite sign in $O$ than in $\sigma_1\setminus \sigma_2$ and~$\sigma_2\setminus \sigma_1$.
The proof of this  is easy. The coefficients in the \emph{linear} dependence among the \emph{vectors} in~$\sigma_1\cup \sigma_2$ are the vector
\[
(\lambda(v) \omega(v))_{ v \in A}.
\] 
To turn this into an \emph{affine} dependence of points involving the origin we simply need to give the origin the coefficient $-\sum_v\lambda(v) \omega(v)$ which is, by hypothesis, negative.
\end{proof}

So, in the light of Lemma~\ref{lemma:polytopal}, to prove Theorem~\ref{thm:triangmany_regular} we simply need to choose weights $\omega_{ij}$ for the diagonals of the polygon with the property that, for each of the linear dependences exhibited in equations~(\ref{eq:triangmany_a}),~(\ref{eq:triangmany_b}),~(\ref{eq:triangmany_c}), and~(\ref{eq:triangmany_d}), the equation
$
\sum_{ij} \omega_{ij} \lambda_{ij} > 0$ holds.

As a first approximation, let $\omega_{ij}=2$ if $ij$ is in $T_0$ and $\omega_{ij}=1$ otherwise. This is good enough for equations~(\ref{eq:triangmany_c}) 
and~(\ref{eq:triangmany_d}) in which all the $\omega$'s in the dependence are $1$ and the sum of the coefficients in the left-hand side is greater than in the right-hand side. It also works for equations~(\ref{eq:triangmany_b}), in which we have
\[
\omega_{pr}=2,\qquad \omega_{qs}=1,\qquad \lambda_{pr}=1,\qquad\lambda_{qs}=1,
\]
so that the sum $\sum_{ij} \omega_{ij} \lambda_{ij}$ for the left-hand side is three, while that of the right-hand side can be $0$, $-1$ or $-2$ depending on the cases for the points $a$ and $b$.

The only (weak) failure is that in equation~(\ref{eq:triangmany_a}) we have 
\[
\lambda_{pr}=1,\qquad \lambda_{qs}=1,\qquad \lambda_{pq}=-1,\qquad\lambda_{rs}=-1
\]
and all the $\omega$'s are $1$, so we get $\sum_{ij} \omega_{ij} \lambda_{ij} = 0$. We solve this by slightly perturbing the~ $\omega$'s. A slight perturbation will not change the correct signs we got for equations~(\ref{eq:triangmany_b}),~(\ref{eq:triangmany_c}), and~(\ref{eq:triangmany_d}). For example, for each $ij$ not in $T_0$ change $\omega_{ij}$ to
\[
\omega_{ij} = 1 + \varepsilon g_{ij}
\]
for a sufficiently small $\varepsilon>0$ and for a vector $(g_{ij})_{ij}$ satisfying
\[
g_{ik}+g_{jl} > \max\{g_{ij}+g_{kl}, g_{il}+g_{jk}\}  \quad \text{for all }i,j,k,l,\ 1\le i<j<k<l\le n+3.
\]
This holds (for example) for $g_{ij}:=(j-i)(n+3+i-j)$.

\subsection{Distinct seed triangulations produce distinct realizations}
Let $\AssII(T)$ denote the $n$-dimensional associahedron obtained with the construction of the previous section starting with a certain triangulation $T$. (This is a slight abuse of notation, since the associahedron depends also in the weight vector $\omega$ that gives the right-hand sides for its inequality definition. Put differently, by 
$\AssII(T)$ we here denote the normal fan rather than the associahedron itself.)
We want to classify the associahedra $\AssII(T)$ by normal isomorphism.

In principle, it looks like we have as many associahedra as there are triangulations (that is, Catalan-many) but that is not the case because, clearly, changing $T$ by a rotation or a reflection does not change the associahedron obtained. The question is whether this is the only operation that preserves $\AssII(T)$, modulo normal isomorphism. To answer this, we look at parallel facets. 

\begin{proposition}
\label{prop:triangmany_parallel}
$\AssII(T_0)$ has exactly $n$ pairs of parallel facets, each pair consisting of (the facet of) one diagonal in $T_0$ and the diagonal obtained from it by a flip in $T_0$.
\end{proposition}

\begin{proof}
$\AssII(T)$ is full-dimensional, so two facets are parallel only if their defining normals are opposite. Since all normals except the ones from the seed triangulation $T_0$ lie in the positive orthant, in every pair of opposite normals one of them comes from the seed triangulation. This is the case of the statement.
\end{proof}

\begin{lemma}
\label{lemma:triangmany_distinct}
Let $Q$ be an $(n+3)$-gon, with $n\ge 2$. For each triangulation $T$ of $Q$ let $B_T$ denote the set consisting of the $n$ diagonals in $T$ plus the $n$ diagonals that can be introduced by a single flip from $T$.
Then for every $T_1\ne T_2$ we have $B_{T_1}\ne B_{T_2}$. 
\end{lemma}

\begin{proof}
Suppose that $T_1$ and $T_2$ had $B_{T_1}= B_{T_2}$. We claim that $T_2$ is obtained from $T_1$ by a set of ``parallel flips''. That is, by choosing a certain subset of diagonals 
of $T_1$ such that no two of them are incident to the same triangle and flipping them simultaneously. This is so because every diagonal $pr$ in $T_2$ but not in $T_1$ intersects a single 
diagonal $qs$ of $T_1$. If $pqr$ and $prs$ were not triangles in $T_2$, then let $a$ be a vertex joined to $pr$ in $T_2$, different from $q$ or $s$. One of  $pa$ and $ra$ intersects the diagonal $qs$ of $T_1$ and one of the edges $pq$, $qr$, $rs$ and $pr$ of $T_1$.

Once we have proved this for $T_2$, the statement is obvious. For every $T_2$ different from $T_1$ but with all its diagonals in $B_{T_1}$ there is a diagonal that we can flip to get one that is not in $B_{T_1}$ (same argument, let $pr$ be a diagonal in $T_2$ but not in $T_1$; let $pq$, $qr$, $rs$ and $pr$ be the other sides of the two triangles of $T_2$ containing $pq$. Flipping any of them, say $pq$, gives a diagonal that crosses $pq$ and $qs$, which are both in $T_1$).
\end{proof}

\begin{corollary}
\label{corollary:triangmany_distinct}
Let $T_1$ and $T_2$ be two triangulations of a convex $(n+3)$-gon. 
Then $\AssII(T_1)$ and $\AssII(T_2)$ 
are normally isomorphic if and only if 
$T_1$ and $T_2$ are equivalent under rotation-reflection.
\end{corollary}

\begin{proof}
If $T_1$ and $T_2$ are equivalent under rotation-reflection then 
the resulting associahedra are clearly the same.
Now suppose that $\AssII (T_1)$ and $\AssII (T_2)$ are normally isomorphic. 
By Lemma \ref{lemma:automorphism_rotation-reflection} the automorphism of the associahedron face lattice induced by the isomorphism corresponds to
a rotation-reflection of the polygon. Now, normal isomorphism preserves the property of a pair of facets being parallel, so this rotation-reflection sends~$B_{T_1}$ to~$B_{T_2}$, and thus~$T_1$ to~$T_2$.
\end{proof}

However, the same is not true if we only look at the set of normal vectors of  $\AssII(T)$:

\begin{proposition}
\label{prop:dualtree}
Let $T_1$ and $T_2$ be two triangulations of the $(n+3$)-gon. Let $A(T_1)$ and $A(T_2)$ be the sets of normal vectors of $\AssII(T_1)$ and $\AssII(T_2)$. Then $A(T_1)$ and $A(T_2)$ are linearly equivalent if, and only if, $T_1$ and $T_2$ have isomorphic dual trees.
\end{proposition}

\begin{proof}
Let $\mathcal T$ be the dual tree of a triangulation $T$. Observe that the edges of $\mathcal T$ correspond bijectively to the inner diagonals in $T$. Moreover, the diagonals of the polygon not used in $T$ correspond bijectively to the possible paths in $\mathcal T$. 
More precisely: for every pair of nodes of $\mathcal T$ (that is, triangles $t_1$ and $t_2$ of $T$) let $p$ (resp. $q$)
be the vertex of $t_1$ (resp. of $t_2$) not visible from $t_2$ (resp. from $t_1$). Then the diagonals of $T$ crossed by $pq$ correspond to the path in $\mathcal T$ joining $t_1$ to $t_2$. 

This means that, if we label the edges of ${\mathcal T}$ with the numbers $1$ through $n$ in the same manner as we labelled the diagonals of $T$ we have that
\[
A(T) = \{-\alpha_i:i\in [n]\} \cup \{\sum_{i\in p} \alpha_i : p \text{ is a path in } {\mathcal T} \}.
\]
So, $A(T)$ can be recovered knowing only $\mathcal T$ as an abstract graph. For the converse, observe that if two trees are not isomorphic then there is no bijection between their edges that sends paths to paths. For example, knowing only the sets of edges that form paths we can identify the (stars of) vertices of the tree as the sets of edges such that every two of them form a path.
\end{proof}

\subsection{Path triangulations, $c$-cluster complexes, and denominator fans in type $A_n$}
\label{sec:combinatorial_c-cluster_complex}

\subsubsection{Associahedra from path triangulations}
Let us call a triangulation of $P_{n+3}$ whose dual tree is a path a \emph{path triangulation}. 
By Proposition~\ref{prop:dualtree}, for a path triangulation $T$ the set of normal vectors to the facets of $\AssII(T)$ equals the almost positive roots in the root system $A_n$, exactly as in the Chapoton--Fomin--Zelevinsky associahedron.  However, these associahedra are not normally equivalent to one another. To analyze this, we encode each path triangulation 
of the $(n+3)$-gon as a sequence of signs $c\in \{+,-\}^{n-1}$, as follows: the coordinates of $c$ correspond to the $n-1$ triangles 
that are not \emph{ears} (that is, are not leaves in the dual path) and we make it a $+$ or a $-$ depending on whether the dual path turns right or left at that vertex (see Figure~\ref{fig:T_c}). We denote $T_c$ the triangulation obtained in this way, for each sequence $c$. 

\begin{figure}[ht]
\begin{center}
\begin{tikzpicture}[scale=1]
\newdimen\L
\L=5cm
\newdimen\K
\K=1.5cm

\def\m{8}

\pgfmathsetmacro{\dm}{360/ \m}
\pgfmathsetmacro{\twicedm}{2*\dm}

\newcommand{\mgon}[3]{

\foreach \x in {0,\dm,...,360} {
     \draw[fill, yshift =#2,xshift = #1,rotate = #3 ] (\x:\K) 
     ;
     
} ;

\draw[ yshift = #2,xshift = #1,rotate = #3] (0:\K)
\foreach \x in {\dm,\twicedm,...,360} {
      -- (\x:\K)
 } -- cycle (90:\K)  ;

}

\newcommand{\diagone}[3]{ 
\draw[ yshift = #2, xshift = #1, rotate = #3] (0:\K)--(2*\dm:\K);
}

\newcommand{\diagtwo}[3]{
\draw[ yshift = #2, xshift = #1, rotate = #3] (0:\K)--(3*\dm:\K);
}

\newcommand{\diagthree}[3]{
\draw[ yshift = #2, xshift = #1, rotate = #3] (0:\K)--(4*\dm:\K);
}

\newcommand{\diagfour}[3]{
\draw[ yshift = #2, xshift = #1, rotate = #3] (0:\K)--(5*\dm:\K);
}

\mgon{0}{0cm}{0};
\diagone{0}{0cm}{3*\dm}
\diagtwo{0}{0cm}{3*\dm}
\diagthree{0}{0cm}{2*\dm}
\diagtwo{0}{0cm}{-2*\dm}
\diagone{0}{0cm}{-1*\dm}

\draw[fill] (-0.82*\K,0) circle(0.05) --(-0.5*\K,-0.4*\K)  circle (0.05)  node[anchor=north]{$-$};
\draw[fill] (-0.5*\K,-0.4*\K)--(-0.2*\K,0.1*\K)  circle (0.05)  node[anchor=south]{$+$};
\draw[fill] (-0.2*\K,0.1*\K)--(0.2*\K,0.1*\K)  circle (0.05)  node[anchor=south]{$+$};
\draw[fill] (0.2*\K,0.1*\K)--(0.5*\K,-0.4*\K)  circle (0.05)  node[anchor=north]{$-$};
\draw[fill] (0.5*\K,-0.4*\K)--(0.82*\K,0) circle(0.05);

\draw (0,-1.4*\K)  node{$c=\{ - ,+,+,- \}$};
\draw (\L,-1.4*\K) node{$T_c$};

\mgon{\L}{0cm}{0};
\diagone{\L}{0cm}{3*\dm}
\diagtwo{\L}{0cm}{3*\dm}
\diagthree{\L}{0cm}{2*\dm}
\diagtwo{\L}{0cm}{-2*\dm}
\diagone{\L}{0cm}{-1*\dm}

\draw[xshift = \L] (-0.82*\K,0) node{$\delta_1$};
\draw[xshift = \L] (-0.45*\K,-0.3*\K) node{$\delta_2$};
\draw[xshift = \L] (-0.15*\K,0.2*\K) node{$\delta_3$};
\draw[xshift = \L] (0.3*\K,0.2*\K) node{$\delta_4$};
\draw[xshift = \L] (0.55*\K,-0.3*\K) node{$\delta_5$};

\end{tikzpicture}
\end{center}
\caption{The triangulation $T_c$ corresponding to the sequence of signs $c=\{-,+,+,-\}$.}
\label{fig:T_c}
\end{figure}
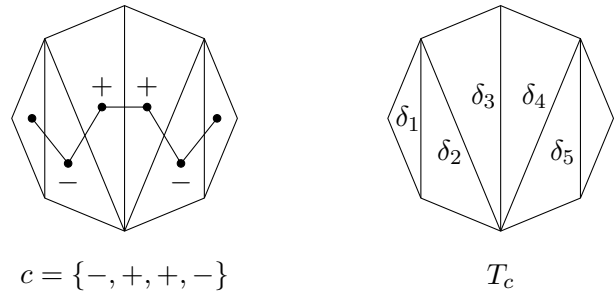

It is clear that every path triangulation can be encoded in this way and that $T_{c_1}$ and $T_{c_2}$ 
are the same (modulo symmetries of the $(n+3)$-gon) if and only if $c_1$ and $c_2$ are the same modulo reflection and reversal. In particular, this gives us exponentially many realizations of the $n$-associahedron with the same set of facet normals as the Chapoton--Fomin--Zelevinsky associahedron:

\begin{corollary}
\label{cor:path_triangulation_associahedra}
Let $T_0$ be a triangulation whose dual tree is a path. Let its diagonals be numbered from $1$ to $n$ in the order they appear in the path. Then,
\begin{compactenum}[\rm(i)]
\item taking $\alpha_i=e_{i+1}-e_i$, we have that 
the set of normal vectors to the facets of $\AssII(T_0)$
is the set of almost positive roots in the root system $A_n$.
\item The number of normally non-isomorphic classes of associahedra obtained in this way is equal to the number of sequences $\{+,-\}^{n-1}$ modulo reflection and reversal.
\end{compactenum}
\end{corollary}

The number of realizations that we get in this way is exactly the same as the number of Hohlweg--Lange associahedra (see Theorem~\ref{thm:classification_HL}). The explanation for this coincidence is in Remark~\ref{rem:cluster-vs-cambrian}.
Nevertheless, the two sets of realizations are almost disjoint; the only common one is the Chapoton--Fomin--Zelevinsky associahedron, obtained in both cases for the sequence that alternates pluses and minuses (Theorem \ref{theo:almost-disjoint}).

\subsubsection{$c$-cluster complexes in type $A_n$}

It turns out that the associahedra  $\AssII(T_0)$ of path triangulations provide a simple combinatorial description of $c$-cluster complexes in type $A_n$ as described by Reading in~\cite{reading_clusters_2007}. 
These complexes are more general than the cluster complexes of Fomin and Zelevinsky~\cite{FZ03}, and have an extra parameter $c$ corresponding to a \textit{Coxeter element}. 

In type $A_n$, Coxeter elements can be represented by a sequence of signs $c\in \{+,-\}^{n-1}$; the corresponding Coxeter element is given by a product of generators $s_1,\dots ,s_n$ in some order such that $s_{i+1}$ comes after $s_i$ if the $i$-th sign in the sequence is positive, and $s_{i+1}$ comes before $s_i$ if the $i$-th sign is negative. Each sequence $c$ induces a single Coxeter element because generators $s_i$ and $s_j$ with $|i-j|\ge 2$ commute.

As in the description of the cluster complex of type $A_n$ in Section~\ref{sec:cluster_complex},
consider the root system of type $A_n$ and the set of almost positive roots $\Phi_{\geq -1}$. In addition, consider a sequence of signs $c\in \{+,-\}^{n-1}$ and let $T_c$ be the corresponding path triangulation.
Label the diagonals of $T_c$ by $\{\delta_1,\dots, \delta_n\}$ in the order they appear in the dual path. 
%
As in the CFZ construction, this gives a natural correspondence between 
the set $\Phi_{\geq-1}$ and the diagonals of $P_{n+3}$: 
We identify the negative simple roots $\{-\alpha_1,\dots,-\alpha_n\}$
with the diagonals $\{\delta_1,\dots, \delta_n\}$, and each positive root
\[
\alpha_{ij}=\alpha_i+\alpha_{i+1}+\dots + \alpha_j, \hspace{1cm} 1\leq i \leq j \leq n,
\]
with the unique diagonal of 
$P_{n+3}$ crossing the (consecutive) diagonals $-\delta_i, -\delta_{i+1}, \dots, -\delta_{j}$.

We say  that two roots $\alpha$ and $\beta$ in $\Phi_{\geq -1}$ are \textit{$c$-compatible} if their corresponding diagonals do not cross. The \textit{$c$-cluster complex} can then be described as the simplicial complex whose faces correspond to sets of almost positive roots that are pairwise $c$-compatible. The maximal simplices in this simplicial complex, which naturally correspond to triangulations of the polygon, are called \textit{$c$-clusters}. For instance, the set 
\[
\{ \alpha_1+\alpha_2+\alpha_3 ,\ 
 \alpha_2+\alpha_3 ,\ 
 \alpha_2+\alpha_3+\alpha_4 ,\ 
 \alpha_3 ,\ 
 -\alpha_5 \}
\]
is a $c$-cluster of type $A_5$ for $c=(-,+,+,-)$ corresponding to the Coxeter element
$s_2s_1s_3s_5s_4$. The reason is that its corresponding diagonals in Figure~\ref{fig:T_c} form a triangulation of the polygon. This algorithm gives a simple combinatorial way of computing $c$-cluster complexes in type $A$. 
The proof that this description of $c$-cluster complexes actually coincides with the original description by Reading follows the two steps $(i)$ and $(ii)$ in the definition of the $c$-compatibility relation in~\cite[Sec.~5]{reading_sortable_2011}. 
As a consequence we obtain:

\begin{proposition}
\label{prop:c-cluster_fan}
The normal fan of the associahedron $\AssII (T_c)$ coincides with the~$c$-cluster fan of type $A_n$.
\end{proposition}

\begin{remark}
\label{rem:cluster-vs-cambrian}
As mentioned in Remark~\ref{rem:gen-associahedra}, the Hohlweg--Lange construction was generalized by Hohlweg--Lange--Thomas to a construction of $c$-generalized associedra (later described in different contexts by Stella~\cite{Stella} and Pilaud--Stump~\cite{PilaudStump}). There is one $c$-generalized associedron for each Coxeter element $c$, and it has the $c$-Cambrian fan~\cite{Reading_cambrian,ReadingSpeyer09} as its normal fan. 

Proposition~\ref{prop:c-cluster_fan} is an analogous result for the $c$-cluster fans, which again exist for each Coxeter element in a finite Coxeter group. In particular, the proposition shows that in type $A$ the $c$-cluster fans are the normal fans of polytopes.
As far as we know, the same is not known in other types, except 
when $c$ is the bipartite Coxeter element. (The $c$-cluster fan is, in this case, the normal fan of the 
Chapoton--Fomin--Zelevinsky generalized associahedron~\cite{CFZ02}.) In fact, in the bipartite case Reading and Speyer have shown that the $c$-Cambrian fan and the $c$-cluster fan are linearly isomorphic~\cite[Thm.~9.1]{ReadingSpeyer09}. In the general case they only show combinatorial isomorphism between them~\cite[Thm.~1.1 and Sec.~5]{ReadingSpeyer09}. 
%
\end{remark}

\subsubsection{Denominator fans in type $A_n$}

For an arbitrary seed triangulation $T$, the normal fan of $\AssII(T)$ can also be interpreted in the language of cluster algebras, as \emph{denominator fans}. For each choice of seed cluster in a cluster algebra, the denominator fan has as rays the denominator vectors of the cluster variables with respect to the seed cluster,
and it has maximal cones spanned by the denominator vectors of variables that form clusters.
The $c$-cluster fan arises as the particular case where the cluster seed corresponds to an \emph{acyclic quiver} associated to a Coxeter element $c$.

Notice that the name ``denominator fan" is a slight abuse of notation since we do not know, a priori, if they are  fans. But in type $A$, clusters correspond to triangulations and the denominator fan with seed triangulation $T$ is nothing but the normal fan of the associahedron~$\AssII (T)$ (see, e.g.,~\cite[Sec.~7]{CeballosPilaud}). Hence, as a consequence of Theorem~\ref{thm:triangmany_fan} and Theorem~\ref{thm:triangmany_regular} we obtain:

\begin{proposition}
\label{prop:denominator_fan}
For cluster algebras of type $A_n$,
\begin{compactenum}[\rm(i)]
\item The denominator fan associated to any triangulation $T$ of a convex $(n+3)$-gon is a complete simplicial fan, and
\item it is the normal fan of a polytope (the associahedron~$\AssII (T)$).
\end{compactenum}
\end{proposition} 

This result suggests a natural generalization of the Santos construction of associahedra to arbitrary finite Coxeter groups: 

\begin{question}
\label{question:CoxeterCatalanAssociahedra}
Given an arbitrary cluster seed in a cluster algebra of finite type 
\begin{compactitem}[$\circ$]
\item is the associated denominator fan a complete simplicial fan?
\item if so, is it the normal fan of a polytope (a generalized associahedron)?
\end{compactitem}
\end{question}

Although this question is phrased in terms of cluster algebras, which deal only with crystallographic root systems, both denominator fans and generalized associahedra make sense in the slightly more general context of finite Coxeter groups. (See~\cite{CeballosPilaud} for an alternative description of denominator vectors in this general set up). 

\begin{question}
If the answer to Question~\ref{question:CoxeterCatalanAssociahedra} is positive, is the classification up to normal isomorphism of Corollary~\ref{corollary:triangmany_distinct} still valid for the generalized associahedra obtained this way? Note that the rotation map on convex polygons can be naturally generalized in the context of finite Coxeter groups (see, e.g., \cite[Sec.~2.2]{CeballosPilaud} or \cite[Sec.~8.3]{CLS13}).   
\end{question}

\section{How many associahedra?}\label{sec:how-many}

We have presented several constructions of the associahedron. 
We call associahedra of types I and II the associahedra $\AssI(\sigma)$ and $\AssII(T)$ studied in the previous two sections.
Associahedra of type I include the Loday (or Shnider--Sternberg, or Rote--Santos--Streinu, or Postnikov, or Buchstaber) associahedron, and both types I and II include the Chapoton--Fomin--Zelevinsky associahedron. They all have pairs of parallel facets while the 
secondary polytope on an $n$-gon does not (Proposition~\ref{prop:par_cont^I}). This implies that
the associahedron as a secondary polytope is never normally isomorphic to any
associahedron of type I or type II. In particular, it is not normally isomorphic to the
Postnikov associahedron or the Chapoton--Fomin--Zelevinsky associahedron.
 
Both types I and II produce exponentially many normally non-isomorphic realizations. 
The number of normally non-equivalent associahedra of type I is asymptotically $2^{n-3}$, while
for type II is asymptotically ${2^{2n+1}}/{\sqrt{\pi n^5}}$. Explicit computations up to dimension $15$ 
are given in Table \ref{table:two-types}. 

\begin{table}[htdp]
\begin{center}
\begin{tabular}{|c|c|c|c|c|c|c|c|c|c|c|c|c|c|c|c|c|}
\hline
$n=$ & 0 & 1 & 2 & 3 & 4 & 5 &6&7&8&9&10&11&12&13&14&15                \\ \hline
$\AssI$ &1&1& 1& 2&3& 6& 10& 20& 36& 72& 136& 272& 528& 1056& 2080& 4160                  \\ \hline
$\AssII$  & 1 & 1 & 1 & 3 & 4 & 12 & 27 & 82& 228& 733& 2282& 7528& 24834& 
		83898& 285357& 983244           \\ \hline
\end{tabular}
\end{center}
\vspace{2mm}
\caption{The number of normally non-isomorphic associahedra
of types I and II up to dimension $15$.}
\label{table:two-types}
\end{table}%

Surprisingly, the realizations of types I and II are almost disjoint:  

\begin{theorem}\label{theo:almost-disjoint}
The only associahedron that is normally isomorphic to both one of type I and one of type II is the Chapoton--Fomin--Zelevinsky associahedron.
\end{theorem}

\begin{proof}
Suppose that a sequence $\sigma \in \{+,-\}^{n-1}$ and a triangulation $T$
produce normally isomorphic associahedra $\AssI(\sigma)$ and $\AssII(T)$.  
By Lemma \ref{lemma:automorphism_rotation-reflection} there is 
no loss of generality in assuming that the bijection between facets
induced by this isomorphism corresponds to the identity map 
on the diagonals of the $(n+3)$-gon.
Also, since normal isomorphism preserves parallelism of facets, the $2n$ diagonals corresponding
to the $n$ pairs of parallel facets are the same in $\AssI(\sigma)$ and $\AssII(T)$. Denote the set of them $B$.

From the perspective of  $\AssII(T)$, $B$ consists of the diagonals of $T$ together with its flips. To analyze $B$ from the perspective of  $\AssI(\sigma)$, we consider the $(n+3)$-gon drawn
in the Hohlweg--Lange fashion (with vertices placed along two $x$-monotone chains, the positive and the negative one, placed in the $x$-order indicated by~$\sigma$). 
By Proposition~\ref{prop:parallel_typeI}, $B$ contains only diagonals between vertices of opposite signs. Knowing this we conclude:
\begin{compactitem}[$\circ$]

\item \emph{Every triangle in $T$ contains a boundary edge in one of the chains. (That is, the dual tree of $T$ is a path).}
Indeed, every triangle contains at least two vertices of the same sign in $\sigma$. The edge joining those two vertices cannot be in $B$, so it is a boundary edge.

\item \emph{The third vertex of each triangle is in the opposite chain. (That is, the dual path of $T$ separates the two chains).}
Otherwise the three vertices of a certain triangle lie in the same chain. This is impossible, because (at least) one of the three edges of each triangle is a diagonal, hence it is in $B$.

\item \emph{No two consecutive boundary edges in one chain are joined to the same vertex in the opposite chain. (That is, the dual tree of $T$ alternates left and right turns)}. Otherwise, let $abp$ and $bcp$ be two triangles in $T$ with $ab$ and $bc$ consecutive boundary edges in one of the chains. Then the flip in $bp$ inserts the edge $ac$, so that $ac\in B$. This is impossible, since $a$ and $c$ are in the same chain.
\end{compactitem}
These three properties imply that $T$ is the snake triangulation, so
$\AssII(T)$ is the Chapoton--Fomin--Zelevinsky associahedron. 
\end{proof}

\end{document}